\documentclass[12pt]{amsart}

\usepackage{amsmath,amssymb,amsfonts,graphics,amsthm}
\usepackage{verbatim}
\usepackage{fullpage}
\usepackage{tikz}
\usepackage{mathabx}

\newcommand{\N}{\mathbb{N}}
\newcommand{\Z}{\mathbb{Z}}
\newcommand{\R}{\mathbb{R}}

\newcommand{\C}{\mathbb{C}}
\newcommand{\F}{\mathbb{F}}
\newcommand{\G}{\mathbb{G}}
\newcommand{\hG}{\widehat{\mathbb{G}}}

\newcommand{\T}{\mathbb{T}}

\newcommand{\mc}{\mathcal}
\newcommand{\id}{\text{id}}

\newcommand{\W}{\mathbb W}

\newcommand{\la}{\langle}
\newcommand{\ra}{\rangle}

\newcommand{\Tr}{\text{Tr}}

\newcommand{\Mor}{\text{Mor}}
\newcommand{\Irr}{\text{Irr}}

\newtheorem{thm}{Theorem}[section]
\newtheorem{cor}[thm]{Corollary}
\newtheorem{lem}[thm]{Lemma}
\newtheorem{prop}[thm]{Proposition}

\theoremstyle{definition}
\newtheorem{defn}[thm]{Definition}

\theoremstyle{remark}
\newtheorem{rem}[thm]{Remark}
\newtheorem{notat}[thm]{Notation}
\newtheorem{quest}[thm]{Question}
\newtheorem{exs}[thm]{Examples}
\newtheorem{ex}[thm]{Example}
\numberwithin{equation}{section}

\begin{document}

\title[Approximation properties for locally compact quantum groups]
{Approximation properties for locally compact quantum groups} 

\date{\today}

\author{Michael Brannan}
\address{Michael Brannan: Department of Mathematics, Texas A\&M University, Mailstop 3368, College Station, TX 77843-3368 USA}
\email{mbrannan@math.tamu.edu}
\maketitle

\begin{abstract}
This is a survey of some aspects of the subject of approximation properties for locally compact quantum groups, based on lectures given at the {\it Topological Quantum Groups} Graduate School, 28 June - 11 July, 2015 in Bed\l{}ewo, Poland.  We begin with a study of the dual notions of amenability and co-amenability, and then consider weakenings of these properties in the form of the Haagerup property and weak amenability.  For discrete quantum groups, the interaction between these properties and various operator algebra approximation properties are investigated.  We also study the connection between central approximation properties for discrete quantum groups and monoidal equivalence for their compact duals.  We finish by discussing the central weak amenability and central Haagerup property for free  quantum groups.     
\end{abstract}

\section{Introduction} \label{intro}

The purpose of this survey is to give readers interested in  operator algebraic quantum group theory a light introduction to the study of approximation properties for locally compact quantum groups.  This survey is not a complete one on this subject by any means.  Instead, we choose in these notes to focus our attention on four important concepts: amenability, co-amenability, the Haagerup property, and weak amenability. 

The study of approximation properties for locally compact groups has a long history, and has many interactions with various branches of mathematics, including probability theory, ergodic theory and dynamics, and (of course) operator algebra theory.  From the perspective of operator algebras, the study of group approximation properties gained central importance with the seminal work of Haagerup \cite{Ha78}, who showed that the reduced free group C$^\ast$-algebras  $C^*_r(\F_k)$ have the metric approximation property, despite being non-nuclear when $k \ge 2$.  Haagerup obtained this result by first showing that the free groups possess  a certain approximation property, now called  {\it the Haagerup property}, which can be thought of as an amenability-like condition.  Building on this work, Cowling and Haagerup \cite{CoHa89} introduced the concept of {\it weak amenability} for locally compact groups and showed that this weak form of amenability is also satisfied by the free groups and real rank one semisimple Lie groups.  The connection to operator algebras here is that a discrete group is weakly amenable if and only if its group von Neumann algebra $\mc L(G)$ has the weak$^*$ completely bounded approximation property \cite{Ha86}.  In recent years, the Haagerup property and weak amenability for groups has also had spectacular application's in Popa's deformation/rigidity program for von Neumman algebras.  See \cite{OzPo10,ChSi13} and the references therein. 

The theory of locally compact quantum groups generalizes the theory of locally compact groups, and contains many interesting examples arising from various free probabilistic, combinatorial, and non-commutative geometric constructions (see  \cite{BaSp,RaWe16,  wo1,wo2,wo3}).  In particular, locally compact quantum groups give rise to interesting and highly non-trivial quantum analogues of group C$^\ast$-/von Neumann algebras, and it is natural to develop a parallel theory of approximation properties for quantum groups and study its interaction with the structure of these corresponding operator algebras.  The main goal of this survey is to develop some of this theory, focusing on both the parallels and the distinctions between the classical and quantum worlds.  At this stage in the game, the theory of approximation properties for general locally compact quantum groups is still in its infancy, with may unresolved open problems and issues.  It is our hope that this survey will inspire readers to pursue questions on this interesting research topic.    

\subsection{Organization of the paper}  The rest of this paper is organized as follows:  Section \ref{prelim} contains some basic facts on locally compact quantum groups that we will use throughout the paper.  In section \ref{amen}, we begin our study of approximation properties with a discussion of amenability and co-amenability for locally compact quantum groups.  Here, we study the dualities between amenability and co-amenability and their connections with  nuclearity/injectivity for quantum group operator algebras.  We also give a proof of the quantum Kesten amenability criterion for discrete quantum groups, and use this to  characterize the amenability of free orthogonal quantum groups and Woronowicz' $SU_q(2)$ quantum group. In section \ref{fourieralg}, we start to  pave the way for studying more general approximation properties for quantum groups by introducing the Fourier algebra and describing the completely bounded Fourier multipliers on this algebra.  In Section \ref{hap/wa}, we use these structures to define the Haagerup property and weak amenability for general locally compact quantum groups, and discuss some basic features of these properties.  In Section \ref{discr}, we restrict our attention to discrete quantum groups for the remainder of the paper.  Here we study the interaction between (weak) amenability and the Haagerup property with various operator algebra approximation properties, like nuclearity, injectivity, the Haagerup approximation property, and the completely bounded approximation property.  In the unimodular case, we prove several equivalences between quantum group approximation properties and operator algebra approximation properties.  In Section \ref{CAP} we introduce {\it central} analogues of (weak) amenability and the Haagerup property, and show that these properties are equivalent to their non-central versions in the unimodular case.  In Section \ref{CAP-monoidal}, we show how central approximation properties for discrete quantum groups are preserved by monoidal equivalences between their compact duals.  Finally, in Section \ref{app} we use the theory developed in the previous sections to study approximation properties for free quantum groups, focusing mainly on the case of free orthogonal quantum groups.

\subsection{Some further reading}

As stated before, we only focus in this paper on a few aspects of the theory of quantum group approximation properties.  In particular, we don't say anything about Kazhdan's property (T) for locally compact quantum groups.  For the very recent and remarkable developments on this topic, we refer the reader to \cite{ChNg15, Ar15, FiMuPa15}.  We also don't consider any of the applications of these approximation properties to structural properties of the operator algebras associated to these quantum groups (e.g. strong solidity, prime factorization results, ...).  See \cite{Is15a,Is15b, dCFY}.  Another aspect of the theory that we do not touch here are {\it external approximation properties} for quantum groups.  For example, there have been very recent and striking developments on questions related to residual finiteness, (inner)linearity,  hyperlinearity, and the Kirchberg factorization property for discrete quantum groups.  See \cite{BaBi10, BaFrSk12, BrCoVe15, Ch, BaWa16}.   

\subsection*{Acknowledgements}
First and foremost, I would like to thank Uwe Franz, Adam Skalski, and Piotr Soltan for organizing such a wonderful graduate school, for inviting me to speak there, and for encouraging me to write this survey.  I would also like to thank \'Etienne Blanchard and Stefaan Vaes for kindly sharing with me their unpublished notes \cite{BlVa02}.

\section{Preliminaries} \label{prelim}

In the following, we write $\otimes$ for the minimal tensor product of C$^\ast$-algebras or  tensor product of Hilbert spaces, $\overline{\otimes}$ for the spatial tensor product of von Neumann algebras and $\odot$ for the algebraic tensor product.  All inner products are taken to be conjugate-linear in the second variable.  Given vectors $\xi,\eta$ in a Hilbert space $H$, we denote by $\omega_{\xi,\eta} \in \mc B(H)_*$ the vector state  $x \mapsto \la  x \xi |\eta\ra$.  If $\xi = \eta$, we simply write $\omega_\xi = \omega_{\xi,\xi}$.

\subsection{Locally compact quantum groups}

Let us first recall from  \cite{KuVa00} and \cite{KuVa03} that 
a  (von Neumann algebraic) \textit{locally compact quantum group (lcqg)} is 
a   quadruple $\G = (M,\Delta, \varphi, \psi)$, where 
$M$ is a von Neumann algebra,  $\Delta:M \to M\overline{\otimes}M$ is a co-associative coproduct, i.e. a 
unital normal $\ast$-homomorphism such that 
\[
(\iota \otimes \Delta) \circ \Delta = (\Delta \otimes \iota) \circ \Delta,
\]
and $\varphi$ and $\psi$ are normal faithful semifinite weights on $M$ such that 
\[ 
 (\iota  \otimes \varphi)\Delta(x)  = \varphi(x) 1 \quad \mbox{and}  \quad
 (\psi \otimes \iota)\Delta(x)  = \psi(x) {1}
\qquad (x\in M^+). 
\]
We call $\varphi$ and $\psi$ the 
\textit{left Haar weight} and the \textit{right Haar weight} of $\G$, respectively, and 
we write  $L^\infty(\G) $  for the quantum group von Neumann algebra $M$.

Associated with  each  locally compact quantum group $\G$,  there is a \emph{reduced} quantum group 
 C*-algebra
 $C_0(\G)\subseteq L^\infty(\G)$  with coproduct  
  \[
 \Delta: x\in C_0(\G) \mapsto \Delta(x) \in M(C_0(\G)\otimes C_0(\G)) \subseteq L^\infty(\G) \overline {\otimes} L^\infty(\G).
 \]
Here we let  $M(B)$ denote the {\it multiplier algebra } of a C*-algebra $B$.
Therefore, $(C_0(\G), \Delta)$ together with $\varphi$ and $\psi$ restricted to $C_0(\G)$  is a   
C*-algebraic  locally compact quantum group in the sense of \cite{KuVa00}.

Using the left Haar weight $\varphi$, we can  apply the GNS construction to obtain an inner product  
\[
\langle {\Lambda_\varphi(x) | \Lambda_\varphi(y)} \rangle = \varphi(y^* x)
\]
on  $ \mathfrak N_{\varphi}=  \{x \in L^\infty(\G) : \varphi(x^*x) < \infty \}$, 
and thus obtain  a Hilbert space $L^2(\G)= L^2(\G, \varphi)$.  Here, $\Lambda_\varphi: \mathfrak N_{\varphi} \to L^2(\G)$ is the canonical injection.
The quantum group von Neumann algebra  $L^\infty(\G)$ is standardly represented on $L^2(\G)$ via  the unital normal
$\ast$-homomorphism $\pi : L^\infty(\G) \to \mc B(L^2(\G))$ satisfying  $\pi(x)\Lambda_\varphi(y) = \Lambda_\varphi(xy)$.
There exists a  (left) \textit{fundamental unitary operator} $W$ on $L^2(\G)\otimes L^2(\G)$, 
which satisfies the pentagonal relation
\[
W_{12} W_{13} W_{23} = W_{23} W_{12}.
\]
Here, and later,
we use the standard leg number notation: $W_{12} = W\otimes I, W_{13}
= \Sigma_{23} W_{12} \Sigma_{23}$ and $W_{23} = 1 \otimes W$, where
 $\Sigma: L^2(\G) \otimes L^2(\G)  \to L^2(\G) \otimes L^2(\G) $ is the flip map.
In this case, the coproduct  $\Delta$ on $L^\infty(\G)$ can be expressed as $\Delta(x) = W^*(1\otimes x)W$.

Let  $L^1(\G) = L^\infty(\G)_*$ be the predual of $L^\infty(\G)$.
Then the   pre-adjoint  of  $\Delta$  induces on $L_1(\G)$ a completely contractive Banach algebra product
\[
\star = \Delta_*: \omega \otimes \omega' \in L^1(\G) \widehat \otimes L^1(\G) \mapsto \omega \star \omega' = (\omega \otimes \omega' ) \Delta \in L^1(\G),
\]
where we let  $\widehat \otimes$ denote the operator space projective tensor product.  We will also use the symbol $\star$ to denote the induced left/right module actions of $L^1(\G)$ on $L^\infty(\G)$:
\[\omega \star x = (\iota \otimes \omega)\Delta(x) \quad \&\quad x \star \omega = (\omega \otimes \iota) \Delta(x) \qquad (x \in L^\infty(\G), \ \omega \in L^1(\G)).\]

Using the Haar weights, one can construct an antipode $S$ on $L^\infty(\G)$ satisfying $(S \otimes \iota)W = W^*$. Since, in general, $S$ is 
unbounded on $L^\infty(\G)$, we  can not use $S$ to define an involution on $L^1(\G)$.
However, we can consider a dense subalgebra $L^1_\sharp (\G)$ of $L^1(\G)$, which is defined to be 
the collection of all  $\omega\in L^1(\G)$ such that there exists $\omega^\sharp \in L^1(\G)$ with
$\la \omega^\sharp, x\ra = \overline{ \la \omega, S(x)^* \ra}$ for each $x\in \mc{D}(S)$.
It is known from  \cite{Ku01} and \cite[Section~2]{KuVa03} that 
 $L^1_\sharp(\G)$ is an involutive Banach algebra with  involution  $\omega\mapsto\omega^\sharp$
 and norm $\|\omega\|_\sharp = \mbox{max}\{\|\omega\|, \|\omega^\sharp\|\}$.  

\subsection{Representations of locally compact quantum groups} 

A \emph{representation} $(\pi, H)$ of a locally compact quantum group $\G$ is  a non-degenerate completely contractive homomorphism  $\pi: L^1(\G) \to \mc B(H)$   whose restriction to $L^1_\sharp(\G)$  is a  $\ast$-homomorphism.
It is shown in \cite[Corollary 2.13]{Ku01}  that each representation $(\pi, H)$ of $\G$ corresponds uniquely to a unitary operator $U_\pi$ 
 in $M(C_0(\G) \otimes \mc K(H))\subseteq L^\infty (\G) \overline{\otimes} \mc B(H)$ such that 
 \[
 (\Delta \otimes \iota)  U_\pi =  U_{\pi, 13}  U_{\pi, 23}.
 \]
The correspondence  is given by 
\[
\pi(\omega)  = (\omega \otimes \iota) U_\pi \in  \mc B(H) \qquad (\omega \in L^1(\G)).
\]  
We call $U_\pi $ the \emph{unitary representation} of $\G$ associated with $\pi$. 

Given two representations $(\pi,H_\pi) $ and $ (\sigma ,H_\sigma)$, we can obtain new representations by forming their tensor product and direct sum.  The unitary operator
\[ 
U_{\pi \otimes \sigma}  :=  U_{\pi,12} U_{\sigma,13} \in M(C_0(\G) \otimes\mc K(H_\pi \otimes H_\sigma))
\]  
determines a representation 
\[
\pi \otimes  \sigma: \omega \in  L^1(\G) \mapsto \mc (\omega \otimes \iota)U_{\pi \otimes \sigma}  \in \mc B(H_\pi \otimes H_\sigma).
\]
We call the representation $({\pi \otimes \sigma}, H_\pi \otimes H_\sigma)$  the  \textit{tensor product}  of $\pi$ and $\sigma$.
The \emph{direct sum}  $(\pi\oplus \sigma, H_\pi \oplus H_\sigma)$ of $\pi$ and $\sigma$ is given by the representation
\[
\pi\oplus \sigma : \omega\in L^1(\G) \to \pi(\omega) \oplus \sigma(\omega)\in \mc B(H_\pi \oplus H_\sigma).
\]

For representations $(\pi,H)$ of $\G$,  one can consider all of the usual notions of (ir)reducibility, intertwiner spaces, unitary equivalence, and so on, that can be considered for representations of general Banach algebras.  In particular, given $\pi,\sigma$ as above, we write 
\begin{align*}\text{Mor}(\pi, \sigma) &= \{T \in \mc B(H_\pi,H_\sigma): T\pi(\omega) = \sigma(\omega)T \qquad (\omega \in L^1(\G)) \}\\
&=  \{T \in \mc B(H_\pi,H_\sigma): (1 \otimes T)U_\pi = U_\sigma(1 \otimes T)\},
\end{align*} 
for the space of intertwiners between $\pi$ and $\sigma$.
 
 The \emph{left regular representation} of $\G$ is defined by 
 \[
 \lambda : \omega \in L^1(\G) \mapsto (\omega \otimes \iota )W\in \mc B(L^2(\G)).
 \]
This is an injective  completely contractive  homomorphism from $L^1(\G)$ into $\mc B(L^2(\G))$ such that 
$\lambda(\omega^\sharp) = \lambda(\omega)^*$ for all $\omega\in L^1_\sharp(\G)$.
Therefore,  $\lambda$ is a $*$-homomorphism when restricted to $L^1_\sharp(\G)$.
The  norm closure of $\lambda(L^1(\G))$ yields  C$^\ast$-algebra $C_0(\hG)$  and the $\sigma$-weak closure of
$\lambda(L^1(\G))$ yields a von Neumann algebra $L^\infty(\hG)$.  
 There is a coproduct on $L^\infty(\hG)$ given by 
 \[
\hat{\Delta}:  \hat x \in L^\infty(\hG) \to \hat\Delta(\hat x) = \hat W^*(1\otimes \hat x) \hat W \in 
L^\infty(\hG)\overline{\otimes} L^\infty(\hG),
\]
 where   $\hat W = \Sigma W^*\Sigma$.
We can find 
Haar weights $\hat \varphi$ and $  \hat \psi$ to turn $\hG = (L^\infty(\hG) , \hat \Delta, \hat \varphi, \hat \psi)$
 into a locally compact quantum group -- the {\it Pontryagin dual} of $\G$.  
 Repeating  this argument for the dual quantum group $\hG$ , we get  the left regular representation 
 \[
 \hat \lambda : \hat\omega \in L^1(\hG) \mapsto (\hat \omega \otimes \iota )\hat W=  (\iota \otimes \hat \omega )W^*
 \in \mc B(L^2(\G)).
 \]
 It turns out that   $C_0(\G)$ and   $L^\infty(\G)$ are just 
  the norm and $\sigma$-weak closure  of  $\hat \lambda(L^1(\hG))$ in $\mc B(L^2(\G))$, respectively.
 This gives us   the  quantum group version of Pontryagin duality $\widehat{\hG}= \G$.


There is  a \emph{universal representation} 
 $\pi_u: L^1(\G) \to \mc B(H_u)$ and
 we   obtain   the  {\it   universal quantum group C$^\ast$-algebra}   $C^u_0(\hG) = \overline{\pi_u(L^1(\G))}^{\|\cdot\|}$.  
 By the universal property,   every representation $\pi: L^1(\G) \to \mc B(H)$  uniquely corresponds to a surjective 
  $\ast$-homomorphism $\hat \pi$ from $C^u_0(\hG)$
  onto the C*-algebra $C_\pi(\hG) = \overline{\pi(L^1(\G))}^{\|\cdot\|}$ such that $\pi = \hat \pi \circ \pi_u .$
  In particular,   the left regular representation  $(\lambda, L^2(\G))$ uniquely determines a surjective $\ast$-homomorphism $\hat \pi_\lambda$ from $C^u_0(\hG)$ onto  $ C_0(\hG)$.  Considering the duality, we can obtain the universal quantum group C$^\ast$-algebra $C^u_0(\G) $ and we denote by $ \pi_{\hat \lambda} : C^u_0(\G) \to C_0(\G)$ the canonical surjective $\ast$-homomorphism.  As shown in \cite{Ku01}, $C^u_0(\G)$ admits a coproduct $\Delta_u:C^u_0(\G) \to M(C^u_0(\G) \otimes C^u_0(\G))$ which turns $(C^u_0(\G), \Delta_u)$ into a C$^\ast$-algebraic locally compact quantum group with left and right Haar weights given by $\varphi \circ \pi_{\hat \lambda}$ and $\psi \circ \pi_{\hat \lambda}$, respectively.    

It is also shown in \cite{Ku01} that the fundamental unitary $W$ of $\G$ admits a ``semi-universal'' version $\text{\reflectbox{$\W$}}\in M(C_0(\G) \otimes C^u_0(\hG))$ which has the property that for each representation $\pi: L^1(\G) \to \mc B(H)$, the corresponding unitary operator $U_\pi$ can be expressed as   
\[(\iota \otimes \hat \pi)\text{\reflectbox{$\W$}} = U_\pi.
\] 
Moreover, $\text{\reflectbox{$\W$}}$ satisfies  the following relations
\begin{align} \label{bichar_relation}
(\Delta \otimes \iota) \text{\reflectbox{$\W$}} = \text{\reflectbox{$\W$}}_{13}\text{\reflectbox{$\W$}}_{23} ~\mbox{and} ~ 
(\iota \otimes \hat \Delta_u) \text{\reflectbox{$\W$}} = \text{\reflectbox{$\W$}}_{13} \text{\reflectbox{$\W$}}_{12}.
\end{align}

\subsection{Compact and discrete quantum groups}  A locally compact quantum group is called 
\textit{compact} if $C_0(\G)$ is unital, or equivalently $\varphi = \psi$ is a $\Delta-$bi-invariant state (up to a scale factor).  In this case, we will write $C(\G)$ for $C_0(\G)$.
 We say $\G$ is \textit{discrete} if $\hG$ is a compact quantum group, or equivalently, if $L^1(\G)$ is unital.  For a self-contained treatment of compact/discrete quantum groups and their duality, we refer to \cite{PoWo90, EfRu94, Ti08}.

Let $\G$ be a discrete quantum group.  Denote by $\Irr(\hG)$ the collection of equivalence classes of finite dimensional irreducible representations of $\hG$.  For each $\pi \in \Irr(\hat  \G)$, select a representative unitary representation $( U^\pi,H_\pi)$.  Then $U^\pi$ can be identified with a unitary matrix 
\[
 U^\pi = [ u_{ij}^\pi] \in M_{n(\pi)}( C(\hG)),
\]
where $n(\pi) = \dim H_\pi$.  The linear subspace $\mc O(\hG) \subseteq C(\hG)$ spanned by $\{ u_{ij}^\pi : \pi \in \Irr(\hG), \ 1 \le i,j \le n(\pi)\}$ is a dense Hopf-$\ast$-subalgebra  of $C(\hG)$.  The  coproduct on $\mc O(\hG)$ is given by  restricting $\hat \Delta $ to $\mc O(\hG)$, and we have 
\[
\hat \Delta  (u^\pi_{ij} ) = \sum_{k=1}^{n(\pi)}u^\pi_{ik} \otimes \hat u^\pi_{kj}.
\]
We call $\mc O(\hG)$ the \textit{algebra of polynomial functions on $\hG$}.  

The Haar state $\hat \varphi = \hat \psi$ is always faithful when restricted
to $\mc O(\hG)$, and  $C^u_0(\hG)$ can be identified with the universal enveloping C$^\ast$-algebra of $\mc O(\hG)$.  This allows us to regard $\mc O(\hG)$ as a dense $\ast$-subalgebra of $C^u_0(\hG)$, and the semi-universal fundamental unitary of 
the discrete quantum group $\G$ is then given by 
\begin{align} \label{bichar_compact} \text{\reflectbox{$\W$}} = \bigoplus_{\pi \in \Irr(\hG)} ( U^\pi)^* =  \bigoplus_{\pi \in \Irr(\hG)} \sum_{1 \le i,j \le d_\alpha} e_{ij}^{\pi}  \otimes   u_{ji}^{\pi \ast}  \in M(C_0( \G) \otimes C^u_0(\hG). 
\end{align}  
In general,  for any discrete quantum group $\G$, we have
\[C_0(\G) =\bigoplus_{\pi \in \Irr(\hG)}^{c_0} \mc B(H_\pi)  \quad \mbox{and} \quad L^\infty(\G) = \prod_{\pi \in \Irr(\hG)} \mc B(H_\pi).
\]  
Denote by $C_c(\G)$ the (algebraic) direct sum $\bigoplus_{\pi \in \Irr(\hG)} \mc B(H_\pi)$.  Then 
$C_c(\G)$  forms a common core for the Haar weights $\varphi$ and $\psi$ on $\G$.   

For each $\pi \in \Irr(\hG)$, denote by $p_\pi \in L^\infty(\G)$ the minimal central projection whose support is $\mc B(H_\pi)$.  
Let $\Tr_{\pi}$
be the canonical trace  on  $\mc B(H_\pi)$ (with $\Tr_{\pi}(1) = n(\pi)$).
One can then find positive invertible matrices $Q_\pi \in \mc B(H_\pi)$ which satisfy 
\[
d(\pi) =\Tr_\pi(Q_\pi) = \Tr_\alpha(Q_\pi^{-1}) \quad \mbox{and} \quad 
\Delta(Q_\pi) = Q_\pi \otimes Q_\pi.
\]  The left and right  Haar weights on $\G$ are given by the formulas 
 \begin{align} \label{eqn:weights}
 \varphi(p_\pi x)= d(\pi) \Tr_\alpha(Q_\pi^{-1} p_\pi x)  \quad \mbox{and} 
  \qquad \psi(p_\pi x) = d(\pi) \Tr_\pi(Q_\pi p_\pi x). 
\end{align}  See \cite[Equations (2.12)-(2.13)]{PoWo90}.
The number $d(\pi)$ appearing above is called the {\it quantum dimension} of $\pi$.  Note that $d(\pi) \ge n(\pi)$, and we have equality for all $\pi$ if and only if $\varphi = \psi$.  In this case, we say that $\G$ is a {\it unimodular discrete quantum group}.

The matrices $Q_{\pi} \in \mc B(H_\pi)$ defined above can be used to describe many other structures associated to $\G$ and $\hG$.  For $\G$, the modular automorphism group $\sigma_t$ turns out to be equal to the scaling group $\tau_t$, and \[\sigma_t(x) = \tau_t(x) = Q_\pi^{-it}xQ_\pi^{it} \qquad (x \in \mc B(H_\pi), \ t \in \R). \]    
On the dual side, the {\it Schur orthogonality relations} for the Haar state $\hat \varphi$ are given by 
 \[\hat \varphi(u^\pi_{ij}(u^{\pi'}_{kl})^*) = \frac{\delta_{\pi\pi'}\delta_{ik}Q_{\pi, lj}}{d(\pi)} \quad \& \quad \hat \varphi((u^\pi_{ij})^*u^{\pi'}_{kl}) = \frac{\delta_{\pi\pi'}\delta_{jl}(Q_{\pi}^{-1})_{ki}}{d(\pi)}. \]
The {\it Woronowicz characters} $(f_z)_{z \in \C}$ of $\hG$ is the family of characters $f_z: \mc O(\hG) \to \C$ given by  
\[f_z(u_{ij}^\pi) = (Q_\pi^z)_{ij} \qquad (\pi \in \Irr(\hG), \ 1 \le  i,j \le n(\pi)).\]     Note that \[f_z\star f_w = f_{z+w} \quad \& \quad f_{z}(x^*) = \overline{f_{-\bar z}(x)} \qquad (z,w \in \C, \ x \in \mc O(\hG)).\]
So that $f_z$ is a $\ast$-character if $z \in i\R$.  The Woronowicz characters can also be used to describe the antipode $\hat S$, the (analytic continuation of the)  scaling group $\hat \tau_z$ and modular group $\hat \sigma_z$, and the unitary antipode $\hat R$:  
\[\hat \tau_z(x) = f_{-iz}\star x\star f_{iz}, \quad \hat S^2(x) = \hat\tau_{-i}(x), \quad \hat\sigma_z(x) = f_{iz}\star x\star f_{iz}, \quad \hat S \circ \hat \tau_{i/2}. \]

\subsection{A few basic examples}  We mention here a few examples of compact quantum groups that we will come to at several points in this paper.  These are the free unitary quantum groups, the free orthogonal quantum groups, and Woronowicz' $SU_q(2)$ quantum group.  We leave it to the reader to check that the following examples do indeed satisfy Woronowicz' axioms for a compact (matrix) quantum group \cite{wo2}.

\begin{notat}
Given a complex $\ast$-algebra $A$ and a matrix $X = [x_{ij}] \in M_N(A)$, we denote by $\bar X$ the matrix $[x_{ij}^*] \in M_N(A)$.
\end{notat}

\begin{defn}[\cite{VaWa96}] Let $N \ge 2$ be an integer and let $F \in \text{GL}_N(\C)$.
\begin{enumerate}
\item   The \textit{free unitary quantum group} $U_F^+$ (with parameter matrix $F$) is the compact quantum group given by the universal C$^\ast$-algebra 
\begin{align} \label{eqn:defining} C^u(U^+_F) = C^*\big(\{v_{ij}\}_{1 \le i,j \le N} \ | \ V = [v_{ij}] \text{ is unitary } \& \ F \bar V F^{-1} \text{ is unitary } \big),\end{align} together with coproduct $\Delta:C^u(U^+_F) \to C^u(U^+_F) \otimes C^u(U_F^+)$ given by
\[\Delta(v_{ij}) = \sum_{k=1}^N v_{ik} \otimes v_{kj} \qquad (1 \le i,j \le N).\]
\item Let $\lambda \in \R\backslash \{0\}$ and assume that $F \bar F = \lambda 1$. The \textit{free orthogonal quantum group} $O_F^+$ (with parameter matrix $F$) is the compact quantum group given by the universal C$^\ast$-algebra 
\begin{align} \label{eqn:defining} C^u(O^+_F) = C^*\big(\{u_{ij}\}_{1 \le i,j \le N} \ | \ U = [u_{ij}] \text{ is unitary } \& \ U = F \bar U F^{-1}\big),\end{align} together with coproduct $\Delta:C^u(O^+_F) \to C^u(O^+_F) \otimes C^u(O_F^+)$ given by
\[\Delta(u_{ij}) = \sum_{k=1}^N u_{ik} \otimes u_{kj} \qquad (1 \le i,j \le N).\]
\end{enumerate}
\end{defn}

\begin{rem}
In the above definition, the coproduct $\Delta$ is defined so that the matrix of generators $V = [v_{ij}]$ (resp. $U = [u_{ij}]$) is always a (fundamental) representation of the compact matrix quantum group $U^+_F$ (resp. $O_F^+$).  
\end{rem}
\begin{rem}
Note that the above definition for $O^+_F$ makes sense for any $F\in \text{GL}_N(\C)$.  The additional condition $F \bar F  \in \R 1$ is equivalent to the requirement that $U$ is always an {\it irreducible} representation of $O^+_F$.  Indeed, Banica \cite{Ba96} showed that $U$ is irreducible if and only if $F\bar F  = \lambda 1$  ($\lambda \in \R$).
\end{rem}

We now define Woronowicz' $SU_q(2)$ quantum group.

\begin{defn}[\cite{wo2}]
Let $q \in [-1,1]\backslash \{0\}$ and let $C(SU_q(2)) = C^*(\alpha, \gamma)$ be the universal unital C$^\ast$-algebra generated by elements $\alpha, \gamma$ subject to the relations which make the $2 \times 2$ matrix 
\[W = [w_{ij}] = \left(\begin{matrix} \alpha & -q\gamma^*\\
\gamma & \alpha^* \end{matrix}\right) \quad \text{unitary}.
\]
Defining $\Delta(w_{ij}) = \sum_{k=1}^2 w_{ik} \otimes w_{kj}$, $\Delta$ extends to a co-product on $C(SU_q(2))$, yielding Woronowicz'  $SU_q(2)$-quantum group.
\end{defn}

\begin{rem}
Actually, the quantum groups $SU_q(2)$ are just spacial cases of free orthogonal quantum groups.  Indeed, for $q \in [-1,1]\backslash \{0\}$, we have \[SU_q(2) = O^+_{F_q} \quad \text{where} \quad F_q = \left( \begin{matrix} 0 &1\\
-q^{-1} & 0 \end{matrix}\right) \in \text{GL}_{2}(\C).\] 
That is, there is a $\ast$-isomorphism $C(O^+_{F_q}) \to C(SU_q(2))$ given by  $u_{ij} \mapsto w_{ij}$.  This fact can be readily checked by comparing generators and relations.
\end{rem}

\section{Amenability and co-amenability} \label{amen}

We begin our discussion of approximation properties of locally compact quantum groups with arguably the most fundamental and natural approximation properties: {\it amenability}. Recall that a locally compact group $G$ is said to be amenable if there exists a finitely additive, left translation-invariant Borel probability measure on $G$.  In functional analytic language, this means the existence of a state $m \in L^\infty(G)^*$ such that 
\[
m(L_xf) = m(f) \qquad (f \in L^\infty(G), \ g \in G),
\] 
where $L_gf: t \mapsto f(g^{-1}t)$ is the left translate of $f$ by $g$.   Such a state $m$ is called a {\it left invariant mean} on $L^\infty(G)$.  The class of amenable groups is quite broad and includes, for example, all compact, abelian, solvable, and nilpotent locally compact groups.  For a very nice introduction to the topic of amenability for locally compact groups, we refer the reader to the monographs \cite{Pa, Ru}.  

Our goal now is to extend he concept of amenability to general locally compact quantum groups.  To do this, we first observe that if $m \in L^\infty(G)^*$ is left-invariant, then $m$ is also invariant under the natural dual right action of $L^1(G)$ on $L^\infty(G)$:
\[
(f, \omega) \mapsto f \star \omega = (\omega \otimes \iota)\Delta(f) \qquad (\omega \in L^1(G), \ f \in L^\infty(G)),
\] 
where $\Delta f (s,t) = f(st)$ is the usual corproduct on $L^\infty(G)$.
Thus, if $m$ is a left-invariant mean on $L^\infty(G)$, then 
\begin{align} \label{top-inv}
m(\omega \otimes \iota)\Delta(f)= m(f) \int_{G}\omega(x)dx = \omega(1)m(f)  \qquad (\omega \in L^1(G), \ f \in L^\infty(G)).
\end{align}
In \cite{Ru}, the above condition \eqref{top-inv} is called {\it topological left invariance} for $m$.  Although topological left-invariance is an a priori weaker notion than left-invariance, it is known \cite[Chapter 1]{Ru} that the existence of a topologically invariant mean on $L^\infty(G)$ is equivalent to amenability of $G$.  Since the notion of topological left-invariance translates directly to the quantum setting, we make the following definition.

\begin{defn}
A lcqg is $\G$ is called {\it amenable} if there exists a state $m \in L^\infty(\G)^*$ such that $m(\omega \otimes \iota)\Delta = \omega(1)m$ for all $\omega \in L_1(\G)$.   We call such a state $m$ a left-invariant mean on $L^\infty(\G)$. 
 \end{defn}

\begin{rem}
There is the obvious notion of a right-invariant mean $m \in L^\infty(\G)^*$, which satisfies $m(\iota \otimes \omega)\Delta = \omega(1)m$ for all $\omega \in L^1(\G)$.  Hence we could define amenability in terms of right-invariant means.  It turns out that our preference for left over right is irrelevant, since an easy calculation shows that a state $m \in L^\infty(\G)^*$ is a left-invariant mean if and only if $m\circ R$ is a right-invariant mean, where $R = S \circ \tau_{i/2}:L^\infty(\G) \to L^\infty(\G)$ is the unitary antipode.  (Recall that $R$ is a $\ast$-anti-automorphism, $R^2 = \iota$, and $\Delta = \sigma (R \otimes R)\Delta R$, where $\sigma$ is the tensor flip map.)   
\end{rem}

\subsection{Some examples of amenable lcqgs}

At this point we should mention at least a few non-trivial examples:  Every compact quantum group is clearly amenable (it's Haar state is an invariant mean), among the discrete quantum groups we will see later on that Woronowicz' $\widehat{SU_q(2)}$ is amenable and $\widehat{O^+_F}$ is amenable if and only if $F \in \text{GL}_2(\C)$.  In fact, it is more generally known that the dual $\widehat{G_q}$ of the $q$-deformation of a simply connected semi-simple compact Lie group $G$ is amenable (see, for example,  \cite{NeTu13} and the references therein).  In the locally compact quantum case, we just content ourselves to mention the work of Caspers \cite{Ca14}, showing that that dual of $SU_q(1,1)$ is amenable.

\subsection{Amenability, nuclearity, and injectivity}
In this section we will show how amenability for a lcqg is intimately related to two of the most fundamental approximation properties in operator algebra theory: injectivity and nuclearity.  Let is begin by recalling these concepts.

\begin{defn}
Let $H$ be a Hilbert space and  $M \subset \mc B(H)$ a von Neumann algebra.  $M$ is called {\it injective} if there exists a conditional expectation $E:B(H) \to M$.  I.e., a unital completely positive map $E$ from $\mc B(H)$ onto $M$ such that $E^2 = E$.
\end{defn}

Although the above definition of injectivity may, at the face of it, appear to not have anything to do with ``approximations'', it turns out to be equivalent to the weak* completely positive approximation property, thanks to some deep work of Connes \cite{co1}.  For future reference, let us recall this concept.

\begin{defn} \label{wcpap}
A von Neumann algebra $M$ has the {\it  weak* completely positive approximation property (w*CPAP)} if there exists a net of finite rank, $\sigma$-weakly continuous, completely positive contractions $T_\alpha:M \to M$ such that 
\[
T_\alpha \to \id_M \qquad \text{pointwise $\sigma$-weakly}.
\] 
\end{defn}

In the C$^\ast$-algebra world, the corresponding analogue of injectivity is nuclearity.  Just like injectivity, nuclearity has many equivalent characterizations (see \cite[Chapter 2]{BrOz08} for example).  In the following, we take the so-called {\it completely positive approximation property} as our definition of nuclearity. 

\begin{defn} \label{cpap}
A C$^\ast$-algebra is {\it nuclear}  if there exists a net of finite rank, completely positive contractions $T_\alpha:A \to A$ such that 
\[
T_\alpha \to \id_A \qquad \text{pointwise in norm}.
\] 
\end{defn}

We now come to a result that connects amenability to injectivity/nuclearity. 

\begin{thm}[\cite{BeTu03}] \label{injectivity}
Let $\G$ be an amenable lcqg with Pontryagin dual $\hG$, and let $\hat\pi:C_0^u(\hG) \to \mc B(H)$ be a non-degenerate $\ast$-representation.  Then $\hat\pi(C_0^u(\hG))$ is nuclear, and $\hat\pi(C_0^u(\hG))''$ is injective.  In particular, $C_0^u(\hG)$ and $C_0(\hG)$ are nuclear, and $L^\infty(\hG)$ is injective. 
\end{thm}

\begin{proof}
Let $\pi: L^1(\G) \to \mc B(H)$ be the associated  non-degenerate $\ast$-representation and let $U = U_\pi \in M(C_0(\G)\otimes \mc K(H))$  be the corresponding unitary representation of $\G$.  First note that the injectivity of $\pi(C_0^u(\hG))''$ is equivalent to the nuclearity of $\pi(C_0^u(\hG))$ by some very deep work of Connes \cite{co1} and Choi-Effros \cite{ChEf77}. Moreover, by a result of Tomiyama \cite{To}, injectivity of $\pi(C_0^u(\hG))''$ is equivalent to that of $\pi(C_0^u(\hG))'$.  It is this latter fact that we will now prove.

To this end, let $m$ be a right-invariant mean on $L^\infty(\G)$, let $\alpha:\mc B(H) \to L^\infty(\G) \overline{\otimes }\mc B(H)$ be given by 
\[
\alpha(x) = U(1 \otimes x)U^* \qquad (x \in \mc B(H)),
\] and define $E:\mc B(H) \to \mc B(H)$ by \[E(x) = (m \otimes \iota)(\alpha(x)) \qquad (x \in \mc B(H)). \]
It is clear from the definitions that $E$ is a unital and completely positive (ucp) map.  We now show that $E(\mc B(H)) = \pi(C_0^u(\hG))'$ and $E^2 = E$.  Since $(\Delta \otimes \iota)U = U_{13}U_{23}$, a simple calculation  shows that 
\[
(\iota \otimes \alpha) \alpha  = (\Delta \otimes \iota) \alpha.
\]
From this equality, we get 
\begin{align*}
\alpha(E(x)) &= (m \otimes \alpha)\alpha(x) = (m \otimes \iota) (\Delta \otimes \iota)(\alpha(x)) \\
&=1 \otimes (m \otimes \iota)\alpha(x) = 1 \otimes E(x) \qquad (x \in \mc B(H)),
\end{align*}
where the last equality follows from the right-invariance of $m$.  As a consequence, \[E^2(x) = (m \otimes \iota)(\alpha(E(x))) = (m \otimes \iota) (1 \otimes E(x)) = E(x) \qquad (x \in \mc B(H)).\]
Next, observe that $E(\mc B(H)) = \{y \in \mc B(H): \alpha(y) = 1 \otimes y\}$.  Indeed, we have already shown the inclusion ``$\subseteq$'', and the inclusion ``$\supseteq$'' follows because the equality $\alpha(y) = 1 \otimes y$ implies that $E(y) = (m \otimes \iota)(1 \otimes y) = y$.
To finish, we note that
\begin{align*}
\{y \in \mc B(H): \alpha(y) = 1 \otimes y\} &= \{y \in \mc B(H): (1 \otimes y)U = U(1 \otimes y)\} \\
&=\{y \in \mc B(H): \pi(\omega)y = y\pi(\omega) \  \forall \omega\in L^1(\G)\} \\
&= \pi(L^1(\G))' = \hat \pi(C^u_0(\hG))'.
\end{align*}
\end{proof}

The above theorem raises the following natural question about the connection between injectivity/nuclearity and  amenability.

\begin{quest} 
For a locally compact quantum group $\G$, does the injectivity of $L^\infty(\hG)$ or the nuclearity of $C_0(\hG)$ (respectively $C^u_0(\hG)$) characterize the amenability of $\G$?
\end{quest}

The answer to the above question (as stated) turns out to be no, even in the classical case.  Let $G$ be any connected locally compact group.  Connes \cite{co1} showed for instance that the group von Neumann algebra $\mc L(G) = L^\infty(\hat G)$ is always injective and that the reduced group C$^\ast$-algebra $C^*_\lambda(G) = C_0(\hat G)$ is always nuclear.  On the other hand, we will see later in Theorem \ref{unimodular} that the answer is yes, provided $\G$ is a unimodular discrete quantum group.  In the general case, one has to impose some additional hypotheses on the conditional expectations $E:B(L^2(\G)) \to L^\infty(\hG)$ in order to yield a characterization of amenability. 

One such additional hypothesis is given by the following definition.

\begin{defn}[Soltan-Viselter \cite{SoVi14}]
A lcqg $\G$ is called {\it quantum injective} if there exists a conditional expectation $E:\mc B(L^2(\G)) \to L^\infty(\G)$ such that $E(L^\infty(\hG)) \subset ZL^\infty(\G)$ (where $ZL^\infty(\G)$ denotes the center of $L^\infty(\G)$).
\end{defn}

It turns out that if one considers quantum injectivity instead of ordinary injectivity, then one obtains a characterization of amenability.  The theorem presented below is from \cite{SoVi14}, but similar results were obtained independently by Crann-Neufang in \cite{CrNe16}.

\begin{thm}[Soltan-Viselter \cite{SoVi14}, Crann-Neufang \cite{CrNe16}]
The following conditions are equivalent for a lcqg $\G$:
\begin{enumerate}
\item $\G$ is amenable.  
\item There exists a conditional expectation $E:\mc B(L^2(\G)) \to L^\infty(\hG)$ with $E(L^\infty(\G)) = \C1$.
\item $\hG$ is quantum injective.  
\end{enumerate}
\end{thm}

\begin{proof}[Sketch]
$(1) \implies (2)$.  Fix a left-invariant mean $m$ and consider the right-regular representation $V \in L^\infty(\hG)'\overline{\otimes }L^\infty(\G)$, which is unitary operator that implements the coproduct on $L^\infty(\G)$ via $\Delta(x) = V(x \otimes 1)V^*$.  Define $E:\mc B(L^2(\G)) \to \mc B(L^2(\G))$ by  \[\omega(E(x)) = m(\omega \otimes \iota)(V(x \otimes 1)V^*) \qquad (x \in \mc B(L^2(\G)), \omega \in \mc B(L^2(\G))_*).\]
Then, as in the proof of Theorem \ref{injectivity}, one can show $E:\mc B(L^2(\G)) \to (L^\infty(\hG)')' = L^\infty(\hG)$ is a conditional expectation.  Moreover \[E(x) = (\iota\otimes m)(V(x \otimes 1)V^*)  =  (\iota\otimes m)\Delta(x) = m(x)1 \qquad (x \in L^\infty(\G)). \]

$(2) \implies (3)$ is trivial.  

$(3) \implies (1)$.
 Let $E$ be a conditional expectation satisfying $(3)$.  Fix a state $\rho \in L^\infty(\hG)^*$ and define a state $m \in L^\infty(\G)^*$ by 
\[
m = \rho \circ E|_{L^\infty(\G)}.
\] 
We claim that $m$ is a left-invariant mean on $L^\infty(\G)$.  To see this, denote by \[\bar \Delta: \mc B(L^2(\G)) \to \mc B(L^2(\G))\overline{\otimes} \mc B(L^2(\G)); \qquad \bar \Delta(x) = W^*(1 \otimes x)W\] the canonical extension to $\mc B(L^2(\G))$ of the coproduct $\Delta$.  Since any conditional expectation $E:\mc B(L^2(\G)) \to L^\infty(\hG)$ is automatically an $L^\infty(\hG)$-bimodule map \cite{To}  (i.e., $E \in \mc {CB}_{L^\infty(\hG)}(\mc B(L^2(\G)))$), we can use a result of Effros-Kishimoto \cite{EfKi87}, which says that $E$ is a pointwise $\sigma$-weak limit of so-called {\it elementary operators}:
\[
E = \lim_\alpha E_\alpha \quad \text{pointwise $\sigma$-weakly, where} \quad E_\alpha(x) = \sum_{i=1}^{n_\alpha} a_{i,\alpha} xb_{i,\alpha} \qquad (a_{i,\alpha}, b_{i,\alpha} \in L^\infty(\hG)').
\] 
Now fix $\omega \in L^1(\G)$ and $x \in L^\infty(\G)$.  Then using the fact that $W \in L^\infty(\G) \overline{\otimes} L^\infty(\hG)$, we get the following $\sigma$-weak limit.
\begin{align*}
E(\omega \otimes \iota)\bar\Delta(x) &=E(\omega \otimes \iota)(W^*(1 \otimes x)W) \\
&=\lim_\alpha \sum_{i=1}^{n_\alpha} (\omega \otimes \iota) (W^*(1 \otimes a_{i,\alpha} xb_{i,\alpha})W) \\
&= (\omega \otimes \iota)W^*(1 \otimes E(x))W \\
&=  (\omega \otimes \iota) (1 \otimes E(x)),
\end{align*}
where the last equality follows from the fact that $E(L^\infty(\G)) \subseteq ZL^\infty(\hG)$.  Applying $\rho$ to both sides of the above equality, we get 
\[
m(\omega\otimes \iota)\Delta(x) = m(x)\omega(1) \qquad (x \in L^\infty (\G)).
\]

\end{proof}

\subsection{Co-amenability}

Up to this point, we have essentially two ways to detect the amenability of a lcqg $\G$.  Namely, finding a left or right-invariant mean on $L^\infty(\G)$, or constructing a suitable conditional expectation $E:\mc B(L^2(\G)) \to L^\infty(\hG)$.  In order to develop further (possibly simpler) means of checking amenability, we now turn to a notion dual to amenability, called co-amenability. 

\begin{defn}
A locally compact quantum group $\G$ is called {\it co-amenable} if there exists a state $\epsilon:C_0(\G) \to \C$ such that $(\iota \otimes \epsilon)\Delta = \iota$.  Such a state $\epsilon$ is called a {\it co-unit} for $C_0(\G)$.  \end{defn}

\begin{lem} If the co-unit $\epsilon$ exists, then  it is unique.
\end{lem}

\begin{proof}
Let $R$ denote the unitary antipode for $L^\infty(\G)$.  Since $\sigma(R\otimes R)\Delta = \Delta R$, we have $(\epsilon R \otimes \iota)\Delta = \iota$.  But then 
\[
\epsilon R = \epsilon R(\iota \otimes \epsilon)\Delta =\epsilon(\epsilon R \otimes \iota)\Delta = \epsilon \implies (\iota \otimes \epsilon )\Delta = (\epsilon \otimes \iota)\Delta = \iota.
\] 
Thus, if $\epsilon'$ is any other co-unit, we have $\epsilon'  = (\epsilon' \otimes \epsilon)\Delta = \epsilon$.  
\end{proof}

In contrast to the notion of amenability, co-amenability turns out to have many equivalent characterizations, several of which can be thought of as approximation properties for the quantum group (or its dual).
\begin{thm} \label{co-amen-characterization}
The following conditions are equivalent for a  lcqg $\G$:
\begin{enumerate}
\item $\G$ is co-amenable.
\item $C_0(\G)^*$ is unital (with respect to the convolution algebra structure induced by $\Delta$).
\item The canonical quotient map $\pi_u:C_0^u(\G) \to C_0(\G)$ is injective.
\item  There is a state $\epsilon \in C_0(\G)^*$ such that  $( \epsilon \otimes \iota) W = 1$. 
\item There is a net $(\xi_i)_{i} \subset L^2(\G)$ of unit vectors such that \[\lim_i\|W(\xi_i \otimes \eta) - (\xi_i \otimes \eta)\| = 0 \qquad (\eta \in L^2(\G)).\]
\item $L^1(\G)$ has a bounded (left/right/two-sided) approximate identity.

\end{enumerate}
\end{thm}

\begin{proof}
$(1) \iff (2) \iff (3)$:  Clearly $\epsilon$ plays the role of the unit in the Banach algebra $C_0(\G)^*$, so (1) and (2) are equivalent.  Note that since \[L^1(\G) \subseteq C_0(\G)^* \subseteq C^u_0(\G)^*\] are identified as closed two-sided ideals \cite{Ku01}, it follows that $C_0(\G)^*$ is unital if and only if $C_0(\G)^* = C_0^u(\G)^*$, which happens if and only if $\pi_u$ is injective.

$(1) \implies (4)$:  From (1), we have 
\[W = (\iota \otimes \epsilon \otimes \iota)(\Delta \otimes \iota)W =(\iota \otimes \epsilon \otimes \iota)W_{13}W_{23} = W(1 \otimes (\epsilon \otimes \iota)W).\] 
Which implies that $(\epsilon \otimes \iota)W = 1$, since $W$ is unitary.  

$(4) \implies (5)$: Let $\epsilon \in C_0(\G)^*$ satisfy $(4)$, and extend it to a state $\tilde \epsilon$ on $\mc B(L^2(\G))$.  Since $L^\infty(\G)$ is in standard form on $L^2(\G)$, we can find a net of 
unit vectors $(\xi_i)_i\subset L^2(\G)$ such that \[\tilde \epsilon(x) = \lim_i \omega_{\xi_i}(x) \qquad (x \in L^\infty(\G)).\]  But then for all $\eta \in L^2(\G)$, we have $(\iota \otimes \omega_\eta)W \in C_0(\G) \subset L^\infty(\G)$, and therefore
\begin{align*}
\|W(\xi_i \otimes \eta) - (\xi_i \otimes \eta)\|^2 &= 2\|\eta\|^2 - 2 \text{Re}(W(\xi_i \otimes \eta)|\xi_i \otimes \eta) \\
&= 2\|\eta\|^2- 2 \text{Re} (\omega_{\xi_i} \otimes \omega_\eta)W\\
&\to 2\|\eta\|^2- 2 \text{Re} (\epsilon \otimes \omega_\eta)W \\
&= 2\|\eta\|^2- 2 \omega_\eta(1) =0.
\end{align*}

$(5) \implies (6)$:  Consider the net $(\omega_{i})_i \subset L^1(\G)$ with $\omega_{i} = \omega_{\xi_i}$.  Then for any $\omega = \omega_{u,v} \in L^1(\G)$ and $x \in L^\infty(\G)$,
\begin{align*}
&|\omega_i\star\omega(x) - \omega(x)|\\
&=  |\la(1 \otimes x)W(\xi_i \otimes u)|W\xi_i \otimes v\ra - \la(1 \otimes x)(\xi_i \otimes u)|\xi_i \otimes v\ra| \\
&=  |\la(1 \otimes x)(W(\xi_i \otimes u)- (\xi_i \otimes u))|W\xi_i \otimes v\ra + \la(1 \otimes x)(\xi_i \otimes u)|W(\xi_i \otimes v)-\xi_i \otimes v\ra| \\ 
&\le \|x\|\|(W(\xi_i \otimes u)- (\xi_i \otimes u))\| \|v\|+ \|x\|\|(W(\xi_i \otimes v)- (\xi_i \otimes v))\| \|u\| \to 0.
\end{align*}
This shows that $(\omega_i)_i$ is a contractive left BAI for $L^1(\G)$.  From this we get that $(\omega_iR)_i$ is a right BAI.  Then it is a standard Banach algebra fact that $(\omega_i + \omega_iR - \omega_iR\omega_i)_i$ is a bounded two-sided approximate identity.

$(6) \implies (1)$:  Without loss of generality, we can assume that our (left) bounded approximate identity $(\omega_i)_i$ converges weak$^\ast$ to $\tilde\epsilon \in L^\infty(\G)^*$.  Put $\epsilon =  \tilde\epsilon|_{C_0(\G)}$.  Then for any $x \in C_0(\G)$, $\omega \in L^1(\G)$, we have
\[
\omega( (\epsilon \otimes \iota)\Delta x) = \epsilon((\iota \otimes \omega)\Delta x) = \lim_i  \omega_i((\iota \otimes \omega)\Delta x) = \lim_i \omega_i\star \omega(x) = \omega(x).
\]
This shows that $(\epsilon \otimes \iota) \Delta = \iota$.  From this formula, it is easy to see that $\epsilon$ is a character on $C_0(\G)$, in particular, $\epsilon$ is a state.
\end{proof}

As an almost immediate consequence of the above theorem, we obtain the following important relationship between co-amenability and amenability with respect to Pontryagin duality.

\begin{cor} \label{co-amen-to-amen}
If $\hG$ is co-amenable, then $\G$ is amenable.
\end{cor}

\begin{proof}
If $\hG$ is co-amenable, then there exists a net of unit vectors $(\xi_i)_i\subset L^2(\G)$ such that  
\[
\|\hat W(\xi_i \otimes \eta) - \xi_i \otimes \eta\| \to 0 \qquad (\eta \in L^2(\G)).\] But since $W = \sigma \hat W \sigma^*$ is unitary, we get 
\[
\|W(\eta \otimes \xi_i) - \eta \otimes \xi_i\| \to 0 \qquad (\eta \in L^2(\G)).\]  
Now let $m \in L^\infty(\G)^*$ be any state which is a  weak$^\ast$-limit point of the net $(\omega_{\xi_i})_i$.  Then we have, for all $\xi,\eta \in L^2(\G)$, 
\begin{align*}
\omega_{\xi,\eta}((\iota \otimes m)\Delta(x)) &= \lim_i \langle W^*(1 \otimes x)W(\xi \otimes \xi_i)|(\eta \otimes \xi_i) \rangle \\
&= \lim_i \langle (1 \otimes x)(\xi \otimes \xi_i)|(\eta \otimes \xi_i) \rangle \\
&= \omega_{\xi,\eta}(1) \lim_{i}\omega_{\xi_i} (x) = \omega_{\xi,\eta}(1) m(x).
\end{align*} 
Therefore $m$ is a left-invariant mean on $L^\infty(\G)$.
\end{proof}

A natural question that arises from the above result is whether the converse of Corollary \ref{co-amen-to-amen} is true for all locally compact quantum groups.

\begin{quest}
Does the amenability of $\G$ imply the co-amenability of $\hG$?
\end{quest}

In its full generality, the above question is still open, however there are some special cases that the answer is known to be yes:  If $\G$ is a classical locally compact group, a theorem of Leptin \cite{Le68} says that $\G$ is amenable if and only if the Fourier algebra $A(\G) \cong L^1(\hG)$ has a bounded approximate identity, which by Theorem \ref{co-amen-characterization} implies that $\hG$ is co-amenable.  If $\G$ is compact, the answer is obviously yes as well, since $L^1(\hG)$ is unital.  If $\G$ is discrete, then the answer is again yes, but this is a highly non-trivial theorem of 
Tomatsu \cite{To06} (also proved independently by Blanchard-Vaes \cite{BlVa02}).          

\begin{thm}[Tomatsu \cite{To06}, Blanchard-Vaes \cite{BlVa02}] \label{disc-amen}
Let $\G$ be an amenable discrete quantum group.  Then $\hG$ is co-amenable.
\end{thm}

In what follows we give a very rough outline of the ideas behind this theorem.  Our arguments follow closely along the lines of the approach of Blanchard-Vaes \cite{BlVa02}.  In the proofs in \cite{BlVa02, To06}, the modular theory of the Haar weights on $L^\infty(\G)$ plays a major role.  For another, very recent proof of  Theorem \ref{disc-amen} that avoids the use of modular theory, we refer the reader to Crann \cite{Cr15}.
 
\begin{proof}[Rough Sketch of Theorem \ref{disc-amen}]
Let $m$ be a left-invariant mean on $L^\infty(\G)$. By a standard Hahn-Banach plus convexity argument, we can write $m$ as a weak$^\ast$ limit of a net of states $(\omega_j)_j \subset L^1(\G)$ such that \begin{align} \label{l1-inv}\|\omega \star \omega_j - \omega(1)\omega_j\|_{L^1(\G)} \to 0 \qquad (\omega \in L^1(\G)).\end{align}
Since $L^\infty(\G)$ is in standard position on $L^2(\G)$, there exists a unique unit vector $\xi_j$ in the positive cone $L^2(\G)^+$ such that $\omega_j = \omega_{\xi_j}$.  Our claim is that co-amenability of $\hG$ will follow if we can show that
\begin{align}\label{as-inv}
\lim_i \|\lambda(\omega)\xi_i - \omega(1)\xi_i\|_{L^2(\G)} = 0 \qquad (\omega \in L^1(\G)).
\end{align}
Indeed, if \eqref{as-inv} holds, then it  follows that the unital linear functional
\[\omega \mapsto \omega(1) \qquad (\omega \in L^1(\G))  \quad \text{satisfies} \quad |\omega(1)| \le \|\lambda(\omega)\|, 
\]
and therefore extends uniquely to a state $\hat \epsilon:C_0(\hG) \to \C$.  But then one readily checks that for $\omega$ belonging to the dense subspace $\{\omega \in L^1(\G): \omega \circ S \in L^1(\G)\}$,
\[
\omega((\hat \epsilon \otimes \iota )\hat W) = \hat \epsilon(\omega \otimes \iota)W^* = \hat \epsilon(\lambda(\omega \circ S)) = \omega(S(1)) = \omega(1).
\]
Thus, $(\hat \epsilon \otimes \iota)\hat W = 1$  and $\hG$ is therefore co-amenable.

To prove \eqref{as-inv}, by linearity and density, it suffices to fix $\pi \in \Irr(\hG)$ and prove \eqref{as-inv} for   $\omega^\pi \in \mc B(H_\pi)_{\ast} = M_{n(\pi)}(\C)_{\ast} \subset L^1(\G)$.  To this end, define  positive normal linear functionals $\eta_j, \mu_j \in M_{n(\pi)}(L^\infty(\G))_*$ by 
\[
\eta_j ([x_{mn}]) = \text{Tr} \otimes \omega_j([x_{mn}]) \quad \& \quad  \mu_j(([x_{mn}]) = \sum_{m,n} \langle \omega_{m,n}^\pi, \omega_j \star x_{mn} \rangle \qquad ([x_{mn}] \in M_{n(\pi)}(L^\infty(\G))),
\]
where $\omega_{ij}^\pi(e_{kl}^\pi) = \delta_{ik}\delta_{jl}$ is the dual basis for the fixed system of matrix units $(e^\pi_{ij})_{1 \le i,j \le n_\pi} \subseteq \mc B(H_\pi)$.  The positivity of $\eta_j$ can readily be verified based on the equation
\[ \sum_{m,n} \langle \omega_{m,n}^\pi, \omega_j \star x_{mn} \rangle =  \sum_{m,n,k} \langle x_{mn}\lambda(\omega_{kn}^\pi)\xi_j|\lambda(\omega_{km}^\pi)\xi_j \rangle.\]

Now, by \eqref{l1-inv}, we have $\|\mu_j - \eta_j\|_{M_{n_\pi}(L^\infty(\G))_*} \to 0$, and consequently we also have  \[\|(Q_\pi \otimes 1)\mu_j(Q_\pi \otimes 1) - (Q_\pi \otimes 1)\eta_j(Q_\pi \otimes 1)\|_{M_{n(\pi)}(L^\infty(\G))_*} \to 0,\]
where $Q_\pi$ is the modular matrix associated to $\pi$.
At this point we'd like to apply the Powers-St{\o}rmer inequality \cite{Ha75} to the above equation, which requires us to identify the unique elements in the tensor product positive cone $(S_2(H_\pi) \otimes L^2(\G))^+$, which implement $(Q_\pi \otimes 1)\mu_j(Q_\pi \otimes 1)$ and $(Q_\pi \otimes 1)\eta_j(Q_\pi \otimes 1)$, respectively.  After some, analysis, the two families of elements in this cone turn out to be
\[(Q_\pi \otimes 1)[\lambda(\omega_{sr}^\pi)\xi_j]_{rs} \quad \& \quad (Q_\pi \otimes \xi_j) \qquad (\text{respectively}). \]  From this, an application of Powers-St{\o}rmer gives
\begin{align*}
& \|(Q_\pi \otimes 1)\mu_j(Q_\pi \otimes 1) - (Q_\pi \otimes 1)\eta_j(Q_\pi \otimes 1)\|_{M_{n_\pi}(L^\infty(\G))_*} \to 0\\
\implies&\|(Q_\pi \otimes 1)[\lambda(\omega_{sr}^\pi)\xi_j]_{rs} -(Q_\pi \otimes \xi_j) \|_{(S_2(H_\pi) \otimes L^2(\G))} \to 0 \\
&\iff \|[\lambda(\omega_{sr}^\pi)\xi_j]_{rs} -1 \otimes \xi_j \|_{(S_2(H_\pi) \otimes L^2(\G))} \to 0
\end{align*}
Applying $\omega^\pi \otimes \iota$ to the last line, where $\omega^\pi \in B(H_\pi)^*$, we get 
\[
\|\lambda (\omega^\pi)\xi_j - \omega^\pi(1)\xi_j\| \to 0.
\]
\end{proof}

\subsection{The Kesten amenability criterion}
We shall close our discussion of amenability and co-amenability by presenting a very handy   criterion for establishing the (non-)amenability of a discrete quantum group, called the {\it Kesten amenability criterion}.   We then use this result to determine precisely when the duals of free orthogonal  quantum groups $O^+_F$ are amenable

Let $\G$ be a discrete quantum group with compact dual $\hG$, and let $U = [u_{ij}] \in M_n(\mc O(\hG))$ be a finite-dimensional representation of $\hG$.  Recall that the {\it character of $U$} is the element $\chi = \chi_U \in \mc O(\hG)$ given by 
\[
\chi = \sum_{i=1}^n u_{ii}\in \mc O(\hG).
\]
Note that that $\chi$ always satisfies the following norm inequalities \[\|\chi\|_{C(\hG)} \le \|\chi\|_{C^u(\hG)} = \hat \epsilon(\chi) = n, \] where $\hat \epsilon: C^u(\hG) \to \C$ is the universal co-unit (which is a character satisfying $\hat \epsilon(u_{ij}) = \delta_{ij}$).  Moreover, the same inequalities hold for the real and imaginary parts $\Re \chi$ and $\Im \chi$.
Recall that $\G$ is called {\it finitely generated} if $\hG$ is a {\it compact matrix quantum group}.  I.e., if there is a finite dimensional unitary representation $U$ whose matrix elements generate $\mc O(\hG)$ as a $\ast$-algebra.  We call such a representation $U$ a {\it fundamental representation} of $\hG$.   The following theorem is the Kesten amenability criterion, which we state in the finitely generated case for simplicity.  (We leave the obvious adaptation to the general case to the reader.)  
\begin{thm}[Banica \cite{Ba99}]
Let $\G$ be a finitely generated  discrete quantum group, let $U = [u_{ij}]_{ij = 1}^n$ be a fundamental (unitary) representation  of $\hG$, and let $\chi = \chi_U \in \mc O(\hG)$ be the character of $U$.  Let $\sigma \subset [-n,n]$ be the spectrum of $\Re\chi \subset C(\hG)$.
Then $\G$ is amenable if and only if $n \in \sigma$.
\end{thm}

\begin{rem}
When $\G = \Gamma$ is a classical  finitely generated discrete group with finite generating set $S$.  The above theorem reduces to the classical  Kesten amenability criterion \cite{Ke59}, which says that $\Gamma$ is amenable if and only if the Markov operator 
\[
T = \frac{1}{2}\sum_{\gamma \in S} \lambda(\gamma) + \lambda(\gamma)^* \in C^*_r(\Gamma) \quad \text{satisfies} \quad  |S| \in \sigma(T).  
\]
\end{rem}

\begin{proof}
If $\G$ is amenable, then $\hG$ is co-amenable and therefore there is a bounded co-unit $\hat \epsilon:C(\hG) \to \C$ which sends $\Re\chi$ to $n$.
Thus $n \in \sigma$.  

Conversely, suppose that $n \in \sigma$.  Put $a_i = 1 - \Re u_{ii} \in C(\hG)$, which is positive.  Then the positive operator $\sum_i a_i$ is not invertible. This means  that there is a sequence $(\xi_k)_k \subset L^2(\G)$ of unit vectors such that $\sum_i \langle a_i \xi_k| \xi_k\rangle \to 0$.  

As each term in the above sum is non-negative, we get 
\begin{align*} \langle a_i \xi_k| \xi_k\rangle \to 0  &\qquad (1 \le i \le n) \\
\iff \|u_{ii}\xi_k - \xi_k\| \to 0 &\qquad (1 \le i \le n), 
\end{align*} where the second line follows from the equality 
\[
\|(1-u_{ii})\xi_k \|^2 =  2 \langle (1-\Re u_{ii})\xi_k|\xi_k\rangle + \|u_{ii}\xi\|^2 -1.
\]

Now let $M_n(C(\hG))$ act on $\C^n\otimes L^2(\G)$ in the canonical way.  As $U$ is unitary, 
\[1 = \| U(e_j \otimes \xi_k)\|^2 = \|\sum_i e_i \otimes u_{ij}\xi_k\|^2 =  \sum_i \|u_{ij}\xi_k\|^2 \qquad (1 \le j \le n),\]
and as $\|u_{ii}\xi_k -\xi_k \| \to 0$ we get $\|u_{ij}\xi_k\| \to 0$ if $i \ne j$.
In summary, this means that \[ u_{ij}\xi_k-\delta_{ij}\xi_k \to 0 \qquad (1 \le i, j \le n).\]  Multiplying the above limit by $u_{il}^*$ and summing over $i$, we also obtain 
\[
\sum_{i}( u_{il}^*u_{ij}\xi_k-\delta_{ij}u_{il}^*\xi_k) = \delta_{lj}\xi_k - u_{jl}^*\xi_k \to 0 \qquad (1 \le l,j \le n).
\]  

Extending linearly and multiplicatively (using the triangle inequality), we get from the above two limits that 
\[
\|x \xi_k - \hat\epsilon(x)\xi_k\| \to 0 \qquad (x \in \mc O(\hG)).
\]
In particular $|\hat \epsilon(x)|  = \lim_k\|x\xi_k\|_{L^2(\G)} \le \|x\|_{C_0(\hG)}$.  I.e., $C_0(\hG)$ has a bounded co-unit, and therefore $\G$ is amenable.
\end{proof}

As an application of the Kesten amenability criterion, we determine when the free orthogonal quantum groups $O^+_F$ are co-amenable.

\begin{thm}[Banica \cite{Ba96}]
Let $n \ge 2$ and $F \in \text{GL}_n(\C)$ be such that $F\bar F = \pm 1$.  Then $O^+_F$ is co-amenable if and only if $n = 2$.  In particular, Woronowicz' $SU_q(2)$ is co-amenable for all $q \in [-1,1] \backslash \{ 0\}$.
\end{thm}

\begin{proof}[Sketch]  Consider the character $\chi = \chi_U$ of the fundamental representation $U$.  Since by definition, we have $\bar U = F^{-1}UF$, it follows from the invariance of characters under equivalence of representations that 
\[
\chi^* = \chi_{\bar U} = \chi_{F^{-1}UF} = \chi.
\]
Thus, to determine when $O^+_F$ is co-amenable, we can try to determine the spectrum $\sigma \subset [-n,n]$ of $\chi \in C_0(O^+_F)$.  To do this, we will identify the spectral measure $\mu$ of $\chi$ with respect to the Haar state $\varphi$ (whose support is $\sigma$), and this can be done by computing the moments of $\mu$.  But in this case, we have
\[
\int_\sigma t^k d\mu = \varphi(\chi^k) = (\varphi \otimes \text{Tr})(U^{\otimes k}) = 
\text{Tr}((\varphi \otimes \iota)U^{\otimes k}) = \dim(\text{Mor}(1, U^{\otimes k})), 
\]
since $(\varphi \otimes \iota)U^{\otimes k}$ is the orthogonal projection onto the space of fixed vectors $\text{Hom}(1, U^{\otimes k})$ of the tensor product representation $U^{\otimes k}$.  

In \cite{Ba96}, Banica showed that for all $n \ge 2$, $\text{Mor}(1, U^{\otimes k}) = \{0\}$ if $k$ is odd, and 
\[
\dim(\text{Mor}(1, U^{\otimes 2m})) = C_{m} = \frac{1}{m +1} {2m\choose m} \qquad (\text{the $m$th Catalan number, } m \in \N).
\]
The  Catalan numbers are the well-known moment sequence for the semicircular measure $\mu$ on the interval $[-2,2]$ (see for example \cite{NiSp06}):
\[
d\mu(t) = \frac{1}{2\pi}\sqrt{4-t^2}1_{[-2,2]}dt. 
\]
Thus $\sigma = [-2,2]$ and therefore $n \in \sigma \iff n = 2$.
\end{proof}

\section{Weakening the notion of amenability} \label{fourieralg}
We now consider various relaxations of the  notion of amenability and co-amenability for locally compact quantum groups.  In fact, in this survey, we will primarily focus on two such notions, namely {\it the Haagerup property} and {\it weak amenability} for locally compact quantum groups.  The primary goal here will be to investigate the interplay between these approximation properties and the structure of certain operator algebras associated to a locally compact quantum group.  

\subsection{Motivation - Amenability and approximate units in the Fourier algebra}

Let $G$ be a locally compact group and let $\lambda_G:G \to \mc U(L^2(G))$ denote the left regular representation.  The {\it Fourier algebra} is the subspace of $A(G) \subset C_0(G)$ consisting of coefficient functions of $\lambda_G$. I.e.,
\[A(G) = \Big\{t \mapsto \phi_{\xi,\eta}(t) =  \la \lambda_G(t^{-1}) \xi|\eta\ra : \xi, \eta \in L^2(G) \Big\}.\]  The Fourier algebra $A(G)$ turns out to be a Banach algebra under the pointwise operations inherited from $C_0(G)$, when it is equipped with the norm
\[
\|\phi\|_{A(G)} = \inf\{\|\xi\|_2\|\eta\|_2: \phi = \phi_{\xi,\eta}\}.  
\]
See \cite{Ey64} for details.  

Since $A(G)$ is a nice commutative Banach algebra that can be associated to any locally compact group $G$, it is natural to ask what types of conditions on $G$ ensure that $A(G)$ has a bounded approximate identity (BAI)?  It turns out that the existence of such a BAI is equivalent to the amenability of $G$.  

\begin{thm}[Leptin \cite{Le68}]
A locally compact group $G$ is amenable if and only if the Banach algebra $A(G)$ has a bounded approximate identity.  

Moreover, if $G$ is amenable, we can always take the BAI $(b_i)_i\subset A(G)$ to consist of normalized {\it positive definite functions} associated to $\lambda_G$.  (I.e., for each $i$, $b_i = \phi_{\xi_i,\xi_i}$ for some unit vector $\xi_i \in L^2(G)$.)
\end{thm}

In other words, if $G$ is not amenable, then the Banach algebra $A(G)$ has no BAI.  However, it could still happen that $A(G)$ has an {\it unbounded} approximate identity $(b_i)_i$ that is  uniformly bounded in some other (smaller) norm on $A(G)$. For example, we could ask for the uniform boundedness of $(b_i)_i$ in the multiplier norm on $A(G)$:
\[
\|b_i\|_{MA(G)} = \sup_{\|a\|_{A(G)} = 1}\|b_ia\|_{A(G)}.
\]
Another scenario one could consider is a non-amenable group $G$ for which there exists a net $(b_i)_i$ of normalized positive definite functions contained in the slightly larger space $C_0(G)$ (instead of $A(G)$) for which a BAI-type condition still holds for $A(G)$.  Namely, 
\[
\|b_ia-a\|_{A(G)} \to 0 \qquad (a \in A(G)).
\]  
It was discovered by Haagerup \cite{Ha78} that the above two scenarios actually hold for the non-amenable free groups $\F_k$ on $k \ge 2$ generators.  With this seminal work of Haagerup, the notion of {\it weak amenability} and {\it the Haagerup property} for groups was born.   Our goal now is to study these approximation properties in some detail in the general context of locally compact quantum groups. 

\subsection{The Fourier algebra of a locally compact quantum group}

We begin by defining the quantum group analogue of the Fourier algebra.  To do this, we consider once again the classical situation:  Let $G$ be a locally compact group and let  
$W \in \mc B(L^2(G) \otimes L^2(G))$ be the associated multiplicative unitary which describes the lcqg structure associated to $G$.  It turns out that $W$ is given by 
\[
W\xi(s,t) = \xi (s,s^{-1}t) \qquad (\xi \in L^2(G \times G) = L^2(G) \otimes L^2(G)),  
\]
and from this it follows that the multiplicative unitary $\hat W  =\Sigma W^*\Sigma$ associated to the Pontryagin dual $\hat G$ satisfies  
\[
\hat W \xi(s,t) = \xi(ts,t) \qquad (\xi \in L^2(G \times G)).
\]
In particular, the dual left regular representation $\hat \lambda:L^1(\hat G) \to C_0(G)$ is given by $\hat \lambda(\omega) = (\omega \otimes \iota)\hat W$.  The following proposition shows that $A(G)$ is nothing but the image of $L^1(\hat G)$  under $\hat \lambda$.

\begin{prop}
For a locally compact group $G$, we have that
\[
\hat \lambda: L^1(\hat G) \to A(G)
\]
is an isometric isomorphism of Banach algebras.
\end{prop}

\begin{proof}
Fix $\omega \in L^1(\hat G)$.  Since $L^\infty(\hat G)$ is in standard position on $L^2(G)$, we can find a pair of vectors $\xi,\eta \in L^2(G)$ such that $\omega = \omega_{\xi,\eta}$ on $L^\infty(\hat G)$ and $\|\omega\|_{L^1(\hat G)} =  \|\xi\|\|\eta\|$.   Now, for $\alpha,\beta \in L^2(G)$, we compute
\begin{align*}
\la \hat \lambda (\omega)\alpha|\beta \ra &= \la \hat W (\xi \otimes \alpha)|(\eta \otimes \beta)\ra \\
&= \int_G \int_G \xi(ts)\alpha(t) \overline{\eta(s)\beta(t)} dsdt \\
&=\int_G \la \lambda(t^{-1})\xi |\eta\ra \alpha(t) \overline{\beta(t)}dt.
\end{align*}
Hence,  $\hat \lambda(\omega_{\xi,\eta}) = \phi_{\xi,\eta} \in A(G)$ and $\|\phi_{\xi,\eta}\|_{A(G)} \le \|\xi\|\|\eta\|$.  Taking the infimum over all $\xi, \eta$ above (noting that the homomorphism $\hat \lambda$ is injective), we get $\|\hat \lambda(\omega)\|_{A(G)} = \|\omega\|_{L^1(\hat G)}$.
\end{proof}

We now turn to the quantum setting and define the Fourier algebra and Fourier multipliers.   

\begin{defn}
For a locally compact quantum group $\G$, the  {\it Fourier algebra} is the subalgebra \[A(\G) := \hat\lambda(L^1(\hG)) \subseteq C_0(\G).\]  For $\omega \in L^1(\hG)$, we write $\|\hat\lambda(\hat\omega)\|_{A(\G)} = \|\hat \omega\|_{L^1(\hG)}$, thus making $A(\G)$ a Banach algebra with respect to the norm $\|\cdot\|_{A(\G)}$.
\end{defn}

\begin{defn}An element $a \in L^\infty(\G)$ is called a {\it (left) Fourier multiplier} if $a A(\G) \subseteq A(\G)$.
\end{defn} 
Note that by the closed graph theorem,  any Fourier multiplier $a \in L^\infty(\G)$ induces a bounded linear map \[m_a:A(\G) \to A(\G); \qquad m_a(b) = ab \qquad (b\in A(\G)).  \] Equivalently, $a$ induces a bounded linear map $L_{*}:L^1(\hG) \to L^1(\hG)$ such that 
\begin{align} \label{eqn:left-mult}
L_*(\omega_1\star \omega_2) = L_*(\omega_1)\star\omega_2 \qquad (\omega_i \in L^1(\hG)),
\end{align}
where $L_*(\omega) = \hat\lambda^{-1}(m_a(\hat \lambda (\omega))).$
In what follows, we will frequently identify a Fourier multiplier $a$ with the corresponding bounded map $m_a$.  We write $MA(\G)$ for the Banach algebra of all Fourier multipliers, and let $M_{cb}A(\G) := MA(\G) \cap \mc {CB}(A(\G))$ be the Banach space of {\it completely bounded (cb) Fourier multipliers}.  Note that $M_{cb}A(\G)$ consists of  precisely those $a \in MA(\G)$ for which the corresponding $\sigma$-weakly continuous map \[L := (L_*)^*:L^\infty(\hG) \to L^\infty(\hG)\] is completely bounded on $L^\infty(\hG)$.  Finally, we call $a \in M_{cb}A(\G)$ a {\it  completely positive multiplier} (or a continuous {\it completely positive definite function} on $\hG$) if the corresponding map $L$ is completely positive.

The following proposition gives some useful equivalent characterizations of those $L \in \mc {CB}_\sigma(L^\infty(\G))$ that come from cb multipliers.

\begin{prop} \label{M_cb-characterization}
Let $L \in \mc {CB}_\sigma(L^\infty(\hG))$ with pre-adjoint $L_* \in \mc {CB}(L^1(\hG))$.  Then the following conditions are equivalent:
\begin{enumerate}
\item $L_*$ satisfies  \eqref{eqn:left-mult}. 
\item $\hat \Delta \circ L = (L \otimes \iota) \circ \hat \Delta$. 
\item There exists a unique $a \in L^\infty(\G)$ such that $(L \otimes \iota) (\hat W) = (1 \otimes a)(\hat W)$.
\item There exists a unique $a \in M_{cb} A(\G)$ such that $a \hat \lambda(\hat\omega) = \hat\lambda(L_*\hat \omega)$  for all $\hat \omega \in L^1(\hG)$.
\end{enumerate}
Moreover, we have $M_{cb}A(\G) \subseteq M(C_0(\G))$, 
\end{prop}

\begin{proof}[Sketch]
The equivalence $(1) \iff (2)$ is clear, and so is $(3) \iff (4) \implies (1)$ (recall that $\hat \lambda(\hat \omega) = (\hat \omega \otimes \iota)\hat W$).  The hard part is $(1) \implies (3)$ and the containment $M_{cb}A(\G) \subseteq  M(C_0(\G))$.  See \cite{JuNeRu09} and \cite{Da11}, respectively, for details.  
\end{proof}

\begin{rem}
Note that Condition (3) in Theorem \ref{M_cb-characterization} implies that for any $a \in M_{cb}A(\G)$, the corresponding map $L = L^{(a)} \in \mc {CB}_\sigma(L^\infty(\hG))$ leaves $C_0(\hG)$ invariant.  Indeed, if $\omega \in \mc K(L^2(\G))^*$, then $(\iota \otimes \omega)\hat W \in C_0(\hG)$ and  
\[
L((\iota \otimes \omega)\hat W) = (\iota \otimes \omega \cdot a)\hat W \in C_0(\hG).
\] The result now follows from the density of $\{(\iota \otimes \omega)\hat W: \omega \in \mc K(L^2(\G))^*\}$ in $C_0(\hG)$.
\end{rem}

\begin{rem}
When $G$ is a locally compact group, a classical unpublished result of Gilbert (see also \cite{Jo92}), gives a useful intrinsic characterization of $M_{cb}A(G)$.  Namely, a function $a \in  M(C_0(G)) = C_b(G)$ belongs to $M_{cb}A(G)$ if and only if there exists a Hilbert space $H$ and  $\alpha,\beta \in C_b(G,H)$ such that 
\begin{align}\label{Gilbert}a(s) = \la \alpha (st)|\beta(t) \ra \qquad (s,t \in G).
\end{align} Moreover, 
\[\|a\|_{M_{cb}A(G)} = \inf \|\alpha\|_{\infty}\|\beta\|_{\infty},\]
where the infimum runs over all $\alpha, \beta, H$ for which \eqref{Gilbert} holds.

In particular, this implies that $B(G)$, the Fourier-Stieltjes algebra of all coefficients of unitary (or even uniformly bounded) representations of $G$ on Hilbert space belong to $M_{cb}A(G)$, with $\|a\|_{M_{cb}A(G)} \le \|a\|_{B(G)}$.

It turns out that there is a quantum analogue of Gilbert's result which was obtained by Daws \cite{Da11}.  We shall not need this result here and simply refer the interested reader to the above reference.  What will be essential for us is the quantum analogue of the containment $B(G) \subseteq M_{cb}A(G)$, which we now state.
\end{rem}

\begin{prop} \label{B(G)-inclusion}
Let $H$ be a Hilbert space and $U \in M(C_0(\G) \otimes \mc K(H))$ a unitary representation of a lcqg $\G$.  Let $\omega \in \mc K(H)^*$ and let \[a = (\iota \otimes \omega)(U^*) \in  M(C_0(\G)).\]  Then $a \in M_{cb}A(\G)$ and $\|a\|_{M_{cb}A(\G)} \le \|\omega\|$.   Moreover, if $\omega \in \mc K(H)^*$ is positive, then  $a$ is a completely positive definite function.

\end{prop}

\begin{proof}
We start off by defining our candidate for the map $L \in \mc {CB}_{\sigma}(L^\infty(\hG))$ induced by $a$ (as in Proposition \ref{M_cb-characterization}).  Define  
\[
L \in \mc {CB}_{\sigma}(L^\infty(\hG), \mc B(L^2(\G))); \qquad L(x) = (\iota \otimes \omega) (U(x \otimes 1)U^*).
\]
It  is clear from the structure of $L$ that $\|L\|_{cb} \le \|\omega\|$ and that $L$ is completely positive if $\omega$ is positive.  Next, we show that $L\in \mc {CB}_\sigma(L^\infty(\hG))$.  To do this, observe that we can write
\begin{align*}	
(L \otimes \iota) \hat W&= (\iota \otimes \iota \otimes \omega)(U_{13}\hat W_{12}U^*_{13}).
\end{align*}
But since $\hat W =  \Sigma W^*\Sigma$, we also have
\begin{align*}	
&U_{13}\hat W_{12}U^*_{13} = U_{13}\Sigma_{12} W^*_{12}\Sigma_{12} U^*_{13} = \Sigma_{12} U_{23}W^*_{12}U^*_{23}\Sigma_{12}\\
&= \Sigma_{12} U_{23}W^*_{12}U^*_{23}W_{12}W_{12}^*\Sigma_{12} \\
&= \Sigma_{12}U_{23}((\Delta \otimes \iota)U^*)W_{12}^*\Sigma_{12} \\
&=\Sigma_{12}U_{23}U_{23}^*U_{13}^*W_{12}^* \Sigma_{12}\\
&=\Sigma_{12}U_{13}^*W_{12}^*\Sigma_{12} = U_{23}^*\Sigma_{12}W_{12}^*\Sigma_{12} =  U_{23}^*\hat W_{12},
\end{align*}
and therefore
\[
(L \otimes \iota) \hat W = (\iota \otimes \iota \otimes\omega)(U_{23}^*\hat W_{12}.) = (1 \otimes a)\hat W. 
\]
From this last equation, it follows that $L$ maps into $L^\infty(\hG)$ (Indeed, just  recall that the set $\{(\iota \otimes \mu)\hat{W}: \mu \in \mc B(L^2(\G))_* \}$ is $\sigma$-weakly dense in $L^\infty(\hG)$).  Moreover, we have that $a \in M_{cb}A(\G)$ implements $L$ by Proposition \ref{M_cb-characterization}.  
\end{proof}

\begin{rem} \label{Daws-Bochner}
In Proposition \ref{B(G)-inclusion} it is shown that whenever we have a unitary representation $U \in M(C_0(\G) \otimes \mc K(H))$ and a state $\omega \in \mc K(H)^*$, the
adjoint map $L^{(a)} \in \mc {CB}_\sigma(L^\infty(\hG))$ corresponding to the multiplier $a = (\iota \otimes \omega)U^* \in M_{cb}A(\G)$ is unital and completely positive.  In fact, the converse statement is also true: if $L^{(a)} \in \mc {CB}_\sigma(L^\infty(\hG))$ is unital and completely positive, then $a \in M_{cb}A(\G)$ is of the form $a = (\iota \otimes \omega)U^*$ for some unitary representation $U \in M(C_0(\G) \otimes \mc K(H))$ and some state $\omega \in \mc K(H)^*$.  See \cite{Da12} for details.  The proof of this result for classical locally compact groups is an easy exercise for the reader (just observe that complete positivity of $a \in M_{cb}A(G)$ implies that $a$ is a continuous positive definite function on $G$).
\end{rem}

With the above remark in mind we will call any  $a \in M_{cb}A(\G)$ of the above form
a {\it (normalized) completely positive definite function} on $\G$.  The general theory of (completely) positive definite functions on lcqgs was studied in detail by Daws and Salmi \cite{DaSa13}.  We refer the reader to this paper for some of the more subtle aspects of this theory in the quantum context.  

\section{The Haagerup property and weak amenability} \label{hap/wa}

We now begin the discussion of our two approximation properties for locally compact quantum groups--the Haagerup property and weak amenability--which can be thought of as weak notions of amenability.  (Or more appropriately, weak notions of co-amenability for their duals, as we shall see below).  

\subsection{The Haagerup property}

The Haagerup property for general  locally compact quantum groups was developed by Daws, Fima, Skalski and White \cite{DaFiSkWh16}.  In what follows, we give a very brief and rather incomplete account.

\begin{defn}
We say that a lcqg $\G$ has the \textbf{Haagerup property} if there exists a bounded approximate identity $(a_i)_i$ for $C_0(\G)$ consisting of normalized completely positive definite functions. 
\end{defn}

In the classical setting, the Haagerup property for a locally compact group is often characterized in representation theoretic terms involving mixing unitary representations and almost invariant vectors.  We shall see now that this characterization carries through in the quantum setting.

\begin{defn}
A unitary representation $U \in M(C_0(\G) \otimes \mc K(H))$ of a lcqg $\G$ is called a {\it mixing representation} if all of its matrix elements belong to $C_0(\G)$.  That is,
\[
(\iota \otimes \omega_{\xi,\eta})U \in C_0(\G) \qquad (\xi, \eta \in H). 
\]
\end{defn}

\begin{ex}
The prototypical example of a mixing representation is the left-regular representation $W \in M(C_0(\G) \otimes C_0(\hG))$.  Indeed, by construction we have $(\iota \otimes \omega)W \in C_0(\G)$ for all $\omega \in \mc K(L^2(\G))^*$.  
\end{ex}

\begin{rem}
Let $\mu$ be a state on the universal C$^\ast$-algebra $C_0^u(\hG)$, associated to the dual quantum group $\hG$.  Then, following \cite{DaFiSkWh16}, we can find a Hilbert space $H$, a unitary representation $U \in M(C_0(\G) \otimes \mc K(H))$ and a cyclic unit vector $\xi \in H$ such that 
\[
\la \pi^u(\omega), \mu \ra = \la \omega, a \ra \qquad (\omega \in L^1(\G))
\]
where $\pi^u:L^1(\G) \to C^u_0(\hG)$ is the universal representation and $a  = (\iota \otimes \omega_{\xi})U \in  M(C_0(\G))$ is the corresponding matrix element of $U$.  

Clearly, if the representation $U$ is mixing, then  $a \in C_0(\G)$.  But in fact the following lemma shows that the converse is true as well.
\end{rem}

\begin{lem}\label{mixing}
Let $\mu \in C_0^u(\hG)^*$ be a state and let $a  = (\iota \otimes \omega_{\xi})U \in  M(C_0(\G))$ be the corresponding matrix element as above.  If $a \in C_0(\G)$, then $U$ is a mixing representation.
\end{lem}

\begin{proof}
Let $\pi:L^1_\sharp(\G) \to \mc B(H)$ be the involutive $\ast$-representation corresponding to $U$.  Then $\xi$ is cyclic for $\pi$, and therefore it suffices to show that
\[
 (\iota \otimes \omega_{\pi(\omega_1)\xi,\pi(\omega_2)\xi})U \in C_0(\G) \qquad (\omega \in L^1_\sharp(\G)).
\]
But for $\omega \in L^1_\sharp(\G)$, 
\begin{align*}
\la\omega,  (\iota \otimes \omega_{\pi(\omega_1)\xi,\pi(\omega_2)\xi})U \ra &= \la \pi(\omega)\pi(\omega_1)\xi| \pi(\omega_2)\eta \ra \\
&= \la\pi(\omega_2^\sharp \star \omega \star \omega_1)\xi|\eta\ra \\
&= \la \omega_2^\sharp \star \omega \star \omega_1, a\ra =  \la \omega, \omega_1 \star a \star \omega_2^\sharp \ra,
\end{align*}
so
\[
 (\iota \otimes \omega_{\pi(\omega_1)\xi,\pi(\omega_2)\xi})U = \omega_1\star a \star \omega_2^\sharp.
\]
Thus our problem reduces to showing that the natural left/right module actions of $L^1_\sharp(\G)$ on $L^\infty(\G)$ leave $C_0(\G)$ invariant.  To do this, we can assume by the Cohen factorization theorem that $\omega_1  = c\cdot \omega_1' $ and  $\omega_2^\sharp = d\cdot\omega_2'$ for some $c, d \in C_0(\G)$ and $\omega_1',\omega_2' \in L^1(\G)$.  But then, 
\begin{align*}
\omega_1 \star a &= (\iota \otimes c \cdot \omega_1') \Delta(a) = (\iota \otimes \omega_1')(\Delta(a) (1 \otimes c)),\\
a\star \omega_2^\sharp &= (d \cdot \omega_2' \otimes \iota)\Delta(a) = (d \cdot\omega_2' \otimes \iota)(\Delta(a)(d \otimes 1)).
\end{align*}
Finally, observing that $\Delta(C_0(\G))(1 \otimes C_0(\G)) \subseteq C_0(\G) \otimes C_0(\G)$ and $\Delta(C_0(\G))(C_0(\G) \otimes 1) \subseteq C_0(\G) \otimes C_0(\G)$ (see \cite[Section 3]{KuVa00}), the result follows.  
\end{proof}

Next, we define the notion of  almost invariant vectors for a unitary representation and study its connection to the Haagerup property.

 \begin{defn}
A unitary representation $U \in M(C_0(\G) \otimes \mc K(H))$ of a lcqg $\G$ is said to have  {\it almost invariant vectors} if there is a net of unit vectors $(\xi_i)_i \subset H$ such that 
\[
\|U(\eta \otimes \xi_i) - (\eta \otimes \xi_i)\| \to 0 \qquad (\eta \in L^2(\G)).
\]
\end{defn} 

\begin{ex}
The left regular representation $W$ of lcqg $\G$ has almost invariant vectors if and only if $\hG$ is co-amenable.  See Theorem \ref{co-amen-characterization}.  
\end{ex}

\begin{ex}
Of course, any representation $U \in M(C_0(\G) \otimes \mc K(H))$ of a lcqg $\G$ with a {\it fixed vector}  (i.e., $ \exists 0 \ne \xi \in H$ such that $U(\eta \otimes \xi) = \eta \otimes \xi$ for all $\eta \in H$) trivially has almost invariant vectors. 
\end{ex}

\begin{thm}
The following conditions are equivalent for a lcqg $\G$:
\begin{enumerate}
\item $\G$ has the Haagerup property.
\item There is a mixing representation $U \in M(C_0(\G) \otimes \mc K(H))$ with almost invariant vectors.
\end{enumerate}
\end{thm} 

\begin{rem}
Before starting the proof, we make one observation.  Recall that the antipode $S:L^\infty(\G) \to L^\infty(\G)$ is a densely defined invertible operator which satisfies the formal identity   
\[
(S \otimes \iota)U = U^*
\] 
for any unitary representation $U \in M(C_0(\G) \otimes \mc K(H))$ of $\G$.  In particular, this identity implies that any coefficient of a unitary representation $b = (\iota \otimes \omega_{\xi,\eta})U$ belongs to the domain $\mc D(S)$ of $S$, and $Sb =  (\iota \otimes \omega_{\xi,\eta})U^* \in M_{cb}A(\G)$.
\end{rem}

\begin{proof}
$(1) \implies (2)$.  Let $(a_i)_i \subseteq M_{cb}A(\G) \cap C_0(\G)$ be a bounded approximate identity for $C_0(\G)$ consisting of normalized completely positive definite functions.  For each $i$, Remark \ref{Daws-Bochner} supplies us with a Hilbert space  $H_i$, a state $\omega_i \in \mc K(H_i)^*$ and a unitary representation $U_i \in M(C_0(\G) \otimes \mc K(H_i))$
such that    
\[
a_i = (\iota \otimes \omega_i)U_i^* \in C_0(\G).
\]
Moreover we may assume that $\omega_i = \omega_{\xi_i}$ for some unit vector $\xi_i \in H_i$ and that $\xi_i$ is cyclic for $U_i$.  Indeed, we can amplify $(U_i,H_i)$ if necessary to obtain $\xi_i$, and then we can restrict $U_i$ to the cyclic representation generated by $\xi_i$. 

Set $b_i = (\iota \otimes \omega_{\xi_i})U_i = S^{-1}a_i.$  Since $S^{-1}:\mc D(S^{-1}) \cap C_0(\G) \to C_0(\G)$ (see \cite[Section 5]{KuVa00}), we have that $b_i \in C_0(\G)$ for each $i$, and therefore each $U_i$ is a mixing representation by Lemma \ref{mixing}. 
Put $H = \oplus_i H_i$ and $U = \bigoplus_i{U_i} \in M(C_0(\G) \otimes \mc K(H))$, where $H = \bigoplus H_i$.  Then $U$ is again a mixing representation, since direct sums of mixing representations are obviously still mixing. Now, since $(a_i)_i$ is a bounded approximate identity for $C_0(\G)$, it follows that $a_i \to 1$ $\sigma$-weakly in $L^\infty(\G)$, and the same is true for $(b_i)_i$.  Indeed, since $(b_i)_i$ is a bounded net, it suffices to show that 
\[
\la  \omega, b_i\ra \to \la \omega, 1 \ra \qquad (\omega \in L^1_\sharp(\G)). 
\]  But for this dense subspace we have 
\[
\la  \omega, b_i\ra = \overline{\la (\omega^\sharp)^* \circ S, S^{-1}a_i  \ra} = \overline{\la (\omega^\sharp)^* , a_i  \ra} \to  \overline{\la (\omega^\sharp)^* ,1  \ra} = \la \omega , 1  \ra.
\] In particular, this $\sigma$-weak convergence implies that for any $\eta \in L^2(\G)$
\begin{align}\label{al-in}
\|U(\eta \otimes \xi_i) - (\eta \otimes \xi_i)\|^2 = 2\|\eta\|^2 - 2\Re\langle \omega_\eta, b_i\rangle \to 2\|\eta\|^2 - 2\Re\langle \omega_\eta, 1 \ra =  0.
\end{align}
Therefore the vectors $(\xi_i)_i \subseteq H$ are almost invariant for the representation $U$. 

$(2) \implies (1)$.  We only outline this direction and refer the reader to the proof of \cite[Theorem 5.5]{DaFiSkWh16} for the full details.  Let $U \in M(C_0(\G) \otimes \mc K(H))$ be a mixing representation, let $(\xi_i)_i \subset H$ be a family of almost invariant vectors for $U$, and let $b_i = (\iota \otimes \omega_{\xi_i})U \in C_0(\G)$.  It then follows from equation \eqref{al-in} that $b_i\to 1 \ \sigma\text{-weakly in  }L^\infty(\G)$.  Working a little harder, one can actually show that $b_ix \to x$ and $xb_i \to x$ weakly in $C_0(\G)$ ($x \in C_0(\G)$), and thus the the convex hull of $(b_i)_i$ contains a bounded approximate 
identity (BAI) for $C_0(\G)$.  In particular, this yields a BAI $(c_j)_j$ for $C_0(\G)$ consisting of elements of the form $c_j = (\iota \otimes \omega_j)U$, where $\omega_j \in \mc K(H)^*$ is a state. Setting $a_j = S(c_j) = (\iota \otimes \omega_j)U^*$, we obtain a net of normalized completely positive definite functions on $\G$ which is also a BAI for $C_0(\G)$.
\end{proof}

\begin{cor}
A locally compact quantum group has the Haagerup property if $\hG$ is co-amenable.
\end{cor}

\begin{proof}
If $\hG$ is co-amenable, then the left-regular representation has almost invariant vectors. 
\end{proof}


\subsection{Weak amenability}

We now define another approximation property for locally compact quantum groups, called  amenability.  For locally compact groups, this concept is due to Cowling and Haagerup \cite{CoHa89}.

\begin{defn} \label{weakamen}
A locally compact quantum group $\G$ is called {\it weakly amenable} if there exists a net $(a_i)_i \subseteq A(\G)$  such that \[\|a_ib-b\|_{A(\G)} \to 0 \quad (b \in A(\G)) \quad \text{and}\quad  \Lambda((a_i)_i):= \limsup_i\|a_i\|_{M_{cb}A(\G)} < \infty.\]
\end{defn}
 
In short, $\G$ is weakly amenable precisely when $A(\G)$ has a left approximate identity that is uniformly bounded in the $\|\cdot\|_{M_{cb}A(\G)}$-norm.  We define the {\it Cowling-Haagerup constant for $\G$} to be the number \[\Lambda_{cb}(\G):= \inf \{\Lambda((a_i)_i): (a_i)_i \subset A(\G) \text{ satisfies Definition \ref{weakamen}}\}.\] 

The next proposition shows that weak amenability is really a weak form of co-amenability for $\hG$.

\begin{prop}
If $\hG$ is co-amenable, then $\G$ is weakly amenable and $\Lambda_{cb}(\G) = 1$.
\end{prop}

\begin{proof}
When $\hG$ is co-amenable, $A(\G)$ has a left bounded approximate identity consisting of normalized completely positive definite functions.
\end{proof}

When $\G$ is a classical group or a discrete quantum group, we know that $\hG$ is co-amenable if and only if $\G$ is amenable.  This equivalence is unknown for general locally compact quantum groups.  A positive solution to this problem would provide an affirmative solution to the following open problem.

\begin{quest}
Let $\G$ be an amenable locally compact quantum group.  Does $\G$ have the Haagerup property and is $\Lambda_{cb}(\G) < \infty$?
\end{quest}

\subsection{Some examples}

\begin{ex}[See \cite{Ha78, CoHa89}]
If $G = \F_n, SL_2(\Z), SL_2(\R), SU(n,1), SO(n,1)$, then $G$ has the Haagerup property and $\Lambda_{cb}(G) = 1$.   
\end{ex}

\begin{ex}[See \cite{Ha86}]
 If $G$ is any connected simple Lie group of real rank greater than $1$ with finite center, then $G$ doe not have the Haagerup property and $\Lambda_{cb}(G) = \infty$.
\end{ex}

\begin{ex}[See \cite{Oz08}]
Let $G$ be a finitely generated word hyperbolic group, then $\Lambda_{cb}(G) < \infty$.
\end{ex}

In the quantum setting, we have the following examples.

\begin{ex}[See \cite{DaFiSkWh16,Fr12}]
Let $(\G_i)_{i \in I}$ be a family of discrete quantum groups.  If each $\G_i$ has the Haagerup property, then their free product $\G = \ast_{i \in I} \G$ has the Haagerup property.  Similarly, if \[\Lambda_{cb}(\G_i) = 1 \quad \text{for each $i$, then} \quad 
\Lambda_{cb}(\G) = 1.
\]
\end{ex}

\begin{ex}[See \cite{Br12, Br13, Fr13, dCFY, Le14}]
Let $\G$ be the discrete quantum group such that $(\hG$ is one of the following compact quantum groups: $O^+_F, U^+_F, S_N^+, H_N^{+(s)}$, where $F \in GL_N(\C)$.  Then  $\G$ has the Haagerup property and $\Lambda_{cb}(\G) = 1$.   We will discuss some of these examples in more detail in Section \ref{app}.  
\end{ex}

In the non-discrete quantum setting, there turns out to be a scarcity of examples.  We highlight the following remarkable result of Caspers \cite{Ca14}.  

\begin{ex}[See \cite{Ca14}]
The (non-discrete) locally compact quantum group $\G = SU_q(1,1)$ is weakly amenable with $\Lambda_{cb}(\G) = 1$, co-amenable, and has the Haagerup property.  
\end{ex}

\section{Discrete quantum groups and operator algebra approximation properties} \label{discr}

In this section we restrict our attention to discrete quantum groups and explore some connections between amenability, weak amenability and the Haagerup property for $\G$, and various approximation properties for C$^\ast$-algebras and von Neumann algebras associated to $\G$.

\subsection{Some operator algebra approximation properties}

Let us start by defining the operator algebra approximation properties that will be relevant to us.

\begin{defn}
Let $M$ be a von Neumann algebra.  We say that $M$ has the {\it weak$^\ast$ completely bounded approximation property (w$^\ast$CBAP)} if there is a net of normal, finite rank cb maps $T_i:M \to M$ such that $T_i \to \iota_M$ pointwise $\sigma$-weakly and $\limsup_i\|T_i\|_{cb} :=C < \infty$.  We call the infimum of all these $C$'s  $\Lambda_{cb}(M)$, the {\it Cowling-Haagerup constant of $M$}. 
\end{defn}

\begin{defn}
Let $A$ be a C$^\ast$-algebra.  We say that $A$ has the {\it completely bounded approximation property (CBAP)} if there is a net of finite rank cb maps $T_i:A \to A$ such that $T_i \to \iota_A$ pointwise in norm and  $\limsup_i\|T_i\|_{cb} :=C < \infty$.  We call the infimum of all these $C$'s  $\Lambda_{cb}(A)$, the {\it Cowling-Haagerup constant of $A$}. 
\end{defn}

\begin{defn}
Let $M$ be a von Neumann algebra equipped with a faithful normal semi-finite weight  $\varphi$.  We say that $(M,\varphi)$ has the {\it Haagerup approximation property} if there exists a net $(\Phi_i)_i$ unital completely positive maps $\Phi_i:M \to M$ such that $\varphi(\Phi_i(x)) \le \varphi(x)$ for all $x \in M_+$, and such that the induced maps 
\[T_i \in \mc B(L^2(M,\varphi); \qquad T_i\Lambda_\varphi(x) = \Lambda_\varphi(\Phi_i(x)) \qquad (x \in \mathfrak N_\varphi)\] are compact and converge strongly to $\iota_{L^2(M,\varphi)}$.
\end{defn}

\begin{rem}
The above definition of the Haagerup property for a von Neumann algebra  is taken from \cite{CaSk15}, where this notion  is studied in detail.  In particular, it is shown there that this property is independent of the choice of weight $\varphi$ on $M$, and therefore one can simply say that {\it $M$ has the Haagerup property}, without reference to a choice of weight.  It should also be mentioned that another approach to the Haagerup property for von Neumann algebras was developed at the same time in \cite{OkTo15}, but this time using the standard form of a von Neumann algebra.  It turns out that both of these approaches to the Haagerup property are equivalent \cite{CaOkSkTo14}.
\end{rem}

\begin{rem}
The above three definitions should be regarded as weaker forms of nuclearity (in the C$^\ast$-context) and injectivity (in the von Neumann context).  Indeed, all three of the above properties are implied by nuclearity/injectivity -- compare with Defintions \ref{wcpap}-\ref{cpap}. 
\end{rem}

We now come to the main theorem of this section, which connects quantum group approximation properties to operator algebra approximation properties in the discrete case.  

\begin{thm}  \label{discrete-AP}
Let $\G$ be a discrete quantum group with compact dual $\hG$.  Then we have the following implications.
\begin{enumerate}
\item $\G$ is amenable $\implies$ $L^\infty(\hG)$ is injective and $C(\hG)$ is nuclear.
\item $\G$ has the Haagerup property $\implies$ $L^\infty(\hG)$ has the Haagerup approximation property.
\item $\G$ is weakly amenable $\implies$ $L^\infty(\hG)$ has the w$^*$CBAP and $C(\hG)$ has the CBAP.  In this case we also have 
\[
\Lambda_{cb}(L^\infty(\hG)), \Lambda_{cb}(C(\hG)) \le \Lambda_{cb}(\G).
\] 
\end{enumerate}
If $\G$ is, in addition,  a {\it unimodular} discrete quantum group, then the reverse implications in $(1)-(3)$ hold, and moreover
\[
\Lambda_{cb}(L^\infty(\hG)) =  \Lambda_{cb}(C(\hG)) = \Lambda_{cb}(\G).
\]
\end{thm}

\subsection{Some preparations}

Before beginning the proof of Theorem \ref{discrete-AP}, let us first recall that for a discrete quantum group, we have
\[
L^\infty(\G) = \prod_{\pi \in \text{Irr}(\hG)} \mc B(H_\pi), \qquad C_0(\G) =\bigoplus_{\pi \in \text{Irr}(\hG)}^{c_0} \mc B(H_\pi),\]
 \[
L^2(\G) = L^2(\hG) = \bigoplus_{\pi \in \text{Irr}(\hG)}^{\ell^2} L^2(\mc B(H_\pi), d(\pi)\text{Tr}(Q_\pi\cdot)),
\] 
and the left regular representation of $\hG$ is given by
\[\hat W= \prod_{\pi \in \text{Irr}(\hG)} \sum_{i,j} u_{ij}^\pi\otimes e_{ij}^\pi,\]
where $(e_{ij}^{\pi})_{1 \le i,j \le n(\pi)}$ denotes a fixed system of matrix units for $\mc B(H_\pi)$ and $U^\pi = [u_{ij}^{\pi}] \in M_{n(\pi)}(C(\G)) = C(\G) \otimes \mc B(H_\pi)$ is a corresponding representative of $\pi$.  Also recall that
\[
C_c(\G)  = \bigoplus_{\pi \in \text{Irr}(\hG)}\mc B(H_\pi) \subset L^\infty(\G).
\] 
denotes the $\sigma$-weakly dense subspace of finitely supported elements.  A standard but crucial observation regarding $C_c(\G)$ is that we have the norm-dense inclusions 
\[
\Lambda_\varphi(C_c(\G)) \subseteq L^2(\G), \quad \& \quad C_c(\G) \subseteq A(\G)
\]
for any discrete quantum group $\G$.  In particular, if $a = \hat \lambda(\omega_{\xi,\eta}) \in A(\G)$ with $\xi,\eta \in L^2(\G)$ and  $\|a\|_{A(\G)} = \|\xi\|_2 \|\eta\|_2$, then for any $\epsilon >0$ we can find $x,y \in C_c(\G)$ such that $\xi' = \Lambda_\varphi(x)$ and $\eta' = \Lambda_\varphi(y)$ satisfy
\[
\|\xi - \xi'\|_2 < \frac{\epsilon}{2\|\eta\|_2} \quad \&\quad \|\eta - \eta'\|_2 < \frac{\epsilon}{2\|\xi\|_2}.
\]
Setting $b= \hat \lambda(\omega_{\xi',\eta'})$, we then have $b \in A(\G)\cap C_c(\G)$ and 
\[
\|a-b\|_{A(\G)} \le \|\xi-\xi'\|_2\|\eta\|_2 + \|\xi\|_2\|\eta-\eta'\|_2 < \epsilon.
\]
Note that the above type of estimate easily implies that any normalized completely positive definite function in  $A(\G)$ can be $A(\G)$-norm approximated by normalized completely positive definite function in $C_c(\G) \cap A(\G)$. 

Now suppose $L = L^{(a)} \in \mc {CB}_\sigma(L^\infty(\hG))$ is the adjoint of a cb multiplier $a = (a^\pi)_{\pi \in \text{Irr}(\hG)} \in M_{cb}A(\G)\subset L^\infty(\G)$, then the defining formula $(L \otimes \iota)\hat W = (1 \otimes a)\hat W$ tells us that
\[
(L \otimes \iota)\sum_{i,j} u^\pi_{ij} \otimes e^\pi_{ij} = \sum_{i,j} u_{ij}^\pi \otimes a^\pi e^\pi_{ij} = \sum_{i,j,k}u^\pi_{ij} \otimes a_{ki}^\pi e^{\pi}_{kj} = \sum_{kj} \sum_{i}a^{\pi}_{ki}u^\pi_{ij} \otimes e^\pi_{kj} \qquad (\pi \in \text{Irr}(\hG)).
\]  
In other words, $L$ acts by the formula
\begin{align} \label{mult-formula}
L(u_{ij}^\pi) = \sum_{k=1}^{n(\pi)}a^\pi_{ik}u^\pi_{kj}  \qquad (1 \le i,j \le n(\pi), \ \pi \in \text{Irr}(\hG)).
\end{align}

\subsection{Proof of Theorem \ref{discrete-AP}}
\subsubsection{For general discrete $\G$}
We give here only a general outline of the proof of the theorem.  For further details, we refer the reader to \cite{Br12,Fr12,DaFiSkWh16,KrRu99}.

{\it Implication} (3): Suppose we have a  net $(a_i)_i \subset A(\G)$ such that 
\[
\|a_ib - b\|_{A(\G)} \to 0 \qquad (b \in A(\G)) \quad \& \quad \limsup_i\|a_i\|_{M_{cb}A(\G)}=C < \infty.
\]
Since $\|a_i\|_{A(\G)} \le \|a_i\|_{M_{cb}A(\G)}$ for each $i$, we can, without loss of generality, assume  that $(a_i)_i \subset C_c(\G)\cap A(\G)$.

Consider now the adjoint maps $(L_i)_i \subset \mc{CB}_\sigma(L^\infty(\hG))$, where $L_i = L^{(a_i)}$.  Then $\limsup_i \|L_i\|_{cb} = C < \infty$, and the condition $\|a_ib - b\|_{A(\G)} \to 0 \quad (b \in A(\G)) $ implies by duality that $L_i \to \iota_{L^\infty(\hG)}$ pointwise $\sigma$-weakly.  Finally formula \eqref{mult-formula} combined with the fact that $a_i \in C_c(\G)$ implies that each $L_i$ is finite rank.   Thus $\Lambda_{cb}(L^\infty(\hG)) < \Lambda_{cb}(\G)$.  For the C$^\ast$-variant, we just need to make the additional observation that $L_i \to \iota_{L^\infty(\hG)}$ pointwise $\sigma$-weakly implies in particular that $a_i \to 1$ pointwise in $L^\infty(\G)$, which then implies that $\|L_ix-x\|_{C(\hG)} \to 0$ for all $x \in \mc O(\hG)$.   Since $\mc O(\hG)$ is norm dense in $C(\hG)$, we see that $L_i|_{C(\hG)} \to \iota_{C(\hG)}$ pointwise in norm.  In particular, $\Lambda_{cb}(C(\hG)) \le \Lambda_{cb}(\G)$.  

{\it Implication} (1):  Thus is just a variant of the above argument, except now we have the additional fact that our net $(a_i)_i \subseteq A(\G)\cap C_c(\G)$ consists of normalized completely positive definite functions.  This additional condition forces $L_i$ to be completely positive and unital, and we are done.

{\it Implication} (2):  This is just another minor modification of ideas in the previous two implications.  Suppose $(a_i)_i \subseteq C_0(\G)$ is a bounded approximate identity consisting of normalized completely positive definite functions.  Then $L_i$ is a unital completely positive map and it is $\hat \varphi$-preserving, since 
\[
\hat \varphi(L_iu_{kl}^\pi) = \sum_r a_{i,kr}^\pi \hat \varphi(u^\pi_{rl}) = \delta_{\pi,\pi_0} = \hat \varphi(u_{kl}^\pi) \qquad (1 \le k,l \le n(\pi), \ \pi \in \text{Irr}(\hG)),
\]
where $\pi_0$ is the (equivalence class of the) trivial representation of $\hG$.  Finally, we observe that the $L^2$-extension $\Phi_i$ of $L_i$ is precisely $a_i \in C_0(\G) \subset \mc K(L^2(\G))  \subset \mc B(L^2(\G))$.  In particular, $\Phi_i$ is compact, and $\Phi_i \to \iota_{L^2(\G)}$ strongly, since $(a_i)_i$ is a bounded approximate identity for $C_0(\G)$. 

\begin{rem}
Before going to the unimodular case, we would like to point out that recently a remarkable extension of the above implication (2) has been obtained by Okayasu, Ozawa and Tomatsu \cite{OkOzTo15}.  Namely, for any locally compact quantum group $\G$ the Haagerup property implies that $L^\infty(\hG)$ has the Haagerup approximation property.  This result is even new for classical locally compact groups.   
\end{rem}

\subsubsection{The unimodular case}  Now suppose that $\G$ is unimodular.  This is equivalent to saying that the dual Haar state $\hat \varphi:L^\infty(\hG) \to \C$ is tracial.  To get the converse of the implications $(1)-(3)$ above, we need a way of manufacturing a completely bounded multiplier on $\G$ from an arbitrary completely bounded map $T \in \mc {CB}(L^\infty(\hG))$.  When $\G$ is unimodular, this is possible thanks to an ``averaging trick''
introduced by Haagerup \cite{Ha86} (for discrete groups) and later extended to unimodular discrete quantum groups by Kraus and Ruan  \cite{KrRu99}.  

The idea of this averaging trick is as follows:  Given a completely bounded map $T:C(\hG) \to L^\infty(\hG)$\footnote{We choose these domains/ranges to be as general as possible}, define $a = a_T \in L^\infty(\G)$ by 
\[
a = (\hat \varphi \otimes \iota)([(T \otimes \iota)\hat W]\hat W^*).
\]
In other words, if we write $a = (a^\pi)_{\pi \in \text{Irr}(\hG)}$ with $a^\pi \in \mc B(H_\pi)$, then 
\begin{align}\label{averaging-formula}a^\pi_{ij} = \sum_{k=1}^{n(\pi)}\hat\varphi((Tu_{ik}^\pi)(u_{jk}^\pi)^*)= \sum_{k=1}^{n(\pi)}\la\Lambda_{\hat\varphi}(Tu_{ik}^\pi)|\Lambda_{\hat\varphi}u_{jk}^\pi)\ra \qquad (1 \le i,j \le n(\pi), \ \pi \in \text{Irr}(\hG)).
\end{align} 
Our claim is that in fact $a \in M_{cb}A(\G)$ with $\|a\|_{M_{cb}A(\G)} \le \|T\|_{cb}$.  To verify this, recall that since $\hat\varphi$ is a faithful tracial state, there exists a unique normal faithful $(\hat \varphi \otimes \hat \varphi)$-preserving conditional expectation
\[
E:L^\infty(\hG) \overline{\otimes} L^\infty(\hG) \to  \hat\Delta(L^\infty(\hG)),
\]
which at the $L^2$-level just corresponds to the orthogonal projection $P:L^2(\hG) \otimes L^2(\hG) \to L^2(\hat\Delta(L^\infty(\hG)))$. In particular, it follows from the Schur orthogonality relations for matrix elements of irreducible unitary representations of $\hG$ that the following formula for $E$ holds:
\begin{align}\label{E-formula}
E(u_{ij}^\pi \otimes u^\sigma_{kl}) = \frac{\delta_{\pi,\sigma}\delta_{j,k}}{n(\pi)}\Big(\sum_{1 \le r \le n(\pi)} u_{ir}^\pi \otimes u_{rl}^\pi\Big) =  \frac{\delta_{\pi,\sigma}\delta_{j,k}}{n(\pi)}\hat \Delta(u^\pi_{il}) \qquad (\pi, \sigma \in\text{Irr}(\hG)).
\end{align}
Now define a linear map
\[
L = L_T:C(\hG) \to L^\infty(\hG); \qquad L = \hat \Delta^{-1}\circ E \circ (T \otimes \iota)\circ \hat \Delta.
\] 
Evidently $L$ is completely bounded with $\|L\|_{cb} \le \|T\|_{cb}$, and $L$ is completely positive/unital/trace-preserving whenever $T$  has these properties.  Let us now evaluate $Lu_{ij}^\pi$, for some $\pi \in \text{Irr}(\hG)$ and $1 \le i,j \le n(\pi)$.  First note that 
\begin{align*}
 (T\otimes \iota) \hat \Delta(u_{ij}^\pi) &=\sum_{k=1}^{n(\pi)}Tu_{ik}^\pi \otimes u_{kj}^\pi \\
&=\sum_{k=1}^{n(\pi)}\sum_{\sigma \in \text{Irr}(\hG)} \sum_{1 \le r,s \le n(\sigma)} n(\sigma) \la \Lambda_{\hat \varphi}(Tu_{ik}^\pi ) | \Lambda_{\hat \varphi}(u^\sigma_{rs}) \ra u_{rs}^\sigma \otimes u_{kj}^\pi
\end{align*}
Where the last equality follows from the fact that $(\sqrt{n(\sigma)}\Lambda_{\hat \varphi}(u_{rs}^\sigma))_{\sigma, r,s}$ is an orthonormal basis for $L^2(\hG)$ (since  $\hat \varphi$ is tracial), and thus 
\[
Tu_{ik}^\pi = \sum_{\sigma \in \text{Irr}(\hG)} \sum_{1 \le r,s \le n(\sigma)} n(\sigma) \la \Lambda_{\hat \varphi}(Tu_{ik}^\pi ) | \Lambda_{\hat \varphi}(u^\sigma_{rs}) \ra u_{rs}^\sigma \qquad (\text{with respect to $L^2$-convergence}).
\]
From these equations we then get
\begin{align*}
Lu_{ij}^\pi& = \hat \Delta^{-1} \circ E\Big(\sum_{k=1}^{n(\pi)}\sum_{\sigma \in \text{Irr}(\hG)} \sum_{1 \le r,s \le n(\sigma)} n(\sigma) \la \Lambda_{\hat \varphi}(Tu_{ik}^\pi ) | \Lambda_{\hat \varphi}(u^\sigma_{rs}) \ra u_{rs}^\sigma \otimes u_{kj}^\pi\Big) \\
&=  \hat \Delta^{-1} \circ  \sum_{k=1}^{n(\pi)}\sum_{\sigma \in \text{Irr}(\hG)} \sum_{1 \le r,s \le n(\sigma)} n(\sigma) \la \Lambda_{\hat \varphi}(Tu_{ik}^\pi ) | \Lambda_{\hat \varphi}(u^\sigma_{rs}) \ra  \frac{\delta_{\pi,\sigma}\delta_{s,k}}{n(\pi)}\hat \Delta(u^\pi_{rj}) \\
&= \sum_{1 \le r,k \le n(\pi)} \la \Lambda_{\hat \varphi}(Tu_{ik}^\pi ) | \Lambda_{\hat \varphi}(u^\sigma_{rk}) \ra u^\pi_{rj} \\
&= \sum_{1 \le r \le n(\pi)}a^\pi_{ir}u^\pi_{rj}
\end{align*}
Comparing this calculation with formula \eqref{mult-formula}, we conclude that $a \in M_{cb}A(\G)$ and $L = L^{(a)}$ is the adjoint of the multiplier $a$, and in particular $\|a\|_{M_{cb}A(\G)} \le \|T\|_{cb}$.

\begin{rem}\label{finite-rank}
In the case that our map $T \in \mc {CB}(C(\hG), L^\infty(\hG))$ is finite rank, we can say even more about the resulting multiplier $a$.  Namely, that $a \in A(\G)$.  Indeed, if $T$ is finite rank, then there exists $x_1, \ldots, x_m \in L^\infty(\hG)$ and $\tau_1, \ldots , \tau_m \in C(\hG)^*$ such that 
\[T(x) = \sum_{1 \le r \le m} \tau_r(x) x_r \qquad (x \in C(\hG)),\]
which yields
\[a^\pi_{ij} = \sum_{1 \le r\le m} \sum_{k = 1}^{n(\pi)}\tau_r(u_{ik}) \la\Lambda_{\hat \varphi} (x_r)|\Lambda_{\hat\varphi}(u_{jk}) \ra.\]
This suggests we define $b_r, c_r \in L^\infty(\G)$ by 
\[b_r^\pi = [\tau_r(u_{ij}^\pi)]_{ij}, \quad  \& \quad c_r^\pi =  [\la\Lambda_{\hat\varphi} (x_r)|\Lambda_{\hat\varphi}(u_{ji}^\pi)\ra]_{ij},  \] 
so that $a = \sum_rb_rc_r$.  Moreover, an easy calculation using the Schur orthogonality relations shows that 
\begin{align*}
\|\Lambda_{\varphi}(c_r)\|_{L^2(\G)}^2 &=\sum_{\pi} n(\pi)\text{Tr}((c_r^\pi)^*c_r^\pi)= \sum_\pi n(\pi) \sum_{1 \le i,j \le n(\pi)} |\la\Lambda_{\hat\varphi} (x_r)|\Lambda_{\hat\varphi}(u_{ji}^\pi)\ra|^2 \\
&=\|\Lambda_{\hat\varphi} x_r\|_{L^2((\hG)}^2  < \infty \qquad (1 \le r \le m).
\end{align*}
In particular, 
\[\|\Lambda_\varphi(a)\|_{L^2(\G)} \le \sum_{1 \le r\le m}\|b_r\|_\infty \|\Lambda_\varphi(c_r)\|_{L^2(\G)}< \infty  \implies a \in \mathfrak N_\varphi.\]
But in the unimodular case, $\mathfrak N_\varphi = \mathfrak N_\varphi^* \subseteq A(\G)$, 
more explicitly, if $p_0 \in L^\infty(\G)$ denotes the rank one projection corresponding to the trivial representation of $\hG$, and $S(=R)$ denotes the unitary antipode for $L^\infty(\G)$,  then also $Sa \in \mathfrak N_\varphi = \mathfrak N_\varphi^*$ and  
\[
a = S^2(a) = S((\iota \otimes \omega_{\Lambda_\varphi(p_0), \Lambda_\varphi((Sa)^*)})W) = (\iota \otimes \omega_{\Lambda_\varphi(p_0), \Lambda_\varphi((Sa)^*)})W^* = \hat \lambda(\omega_{\Lambda_\varphi(p_0), \Lambda_\varphi((Sa)^*)}) \in A(\G).
\]  For details, see for example \cite[Proposition 3.9]{BrRu14}.
\end{rem}

Let us now use our averaging trick to sketch  the the following implications:
\begin{enumerate}
\item[(a)] $L^\infty(\hG)$ injective $\implies$ $\G$ is amenable.
\item[(b)] $\Lambda_{cb}(L^\infty(\hG)) \ge \Lambda_{cb}(\G)$.
\item[(c)] $L^\infty(\hG)$ has the Haagerup property $\implies$ $\G$ has the Haagerup property.
\end{enumerate} 
(Note: The C$^\ast$-variants of (a),(b) require only minor modifications which are left to the reader).

(a).   Since $L^\infty(\hG)$ is injective,  there exists a net of normal finite rank UCP maps $(T_i) \subset \mc {CB}_\sigma(L^\infty(\hG))$ such that 
$T_i \to \iota_{L^\infty(\hG)}$ pointwise $\sigma$-weakly.  Let $(a_i)_i \subset A(\G)$ 
be the corresponding net of normalized completely positive definite functions obtained by the averaging trick.  Since $T_i \to \iota_{L^\infty(\hG)}$ pointwise $\sigma$-weakly, a comparison with formula \eqref{averaging-formula} shows that $a_i^\pi \to 1 \in \mc B(H_\pi)$ in norm for all $\pi \in \text{Irr}(\hG)$.  In particular, 
\[
\|a_ib - b\|_{A(\G)} \to 0 \qquad (b \in C_c(\G)).
\]  But since $C_c(\G)$ is dense in $A(\G)$, this implies that $(a_i)_i$ is a bounded approximate identity for $A(\G)$.

(b).   The argument is almost identical to (a):  If there exists a net of normal finite rank completely bounded maps $(T_i) \subset \mc {CB}_\sigma(L^\infty(\hG))$ such that 
$T_i \to \iota_{L^\infty(\hG)}$ pointwise $\sigma$-weakly, the corresponding net $(a_i)_i \subseteq A(\G)$ will be a left approximate identity for $A(\G)$ with 
\[
\limsup_i\|a_i\|_{M_{cb}A(\G)} \le \limsup_{i}\|T_i\|_{cb}.
\]
Thus $\Lambda_{cb}(\G) \le \Lambda_{cb}(L^\infty(\hG))$.

(c).  This is yet another variant of (a):  Since $L^\infty(\hG)$ has the Haagerup approximation property, then we can find a net of UCP $\hat \varphi$-preserving maps $(T_i) \subset \mc {CB}_\sigma(L^\infty(\hG))$ such that the corresponding $L^2$-extensions $\Phi_i$ are compact, and $\Phi_i \to \iota_{L^2(\hG)}$  pointwise in norm.  Again, the point-norm convergence of $\Phi$, when compared with formula \eqref{averaging-formula}, tells us that $a^\pi_i \to 1$
for all $\pi \in \text{Irr}(\hG)$, and therefore 
\[
\|a_ib - b\|_{C_0(\G)} \to 0 \qquad (b \in C_0(\G)).
\]
Moroever, the $L^2$-compactness of $\Phi_i$ turns out  to imply each $a_i \in C_0(\G)$. See \cite[Theorem 6.7]{DaFiSkWh16} for details. 

\begin{rem} From Theorem \ref{discrete-AP}, we see that for unimodular discrete quantum
groups, there is a tight connection between approximation properties of $\G$ in terms of completely bounded Fourier multipliers and various approximation properties for $C(\hG)$ and $L^\infty(\hG)$.  In the non-unimodular case, this connection only goes one way and it is an important problem in quantum group theory to understand to what extent the unimodular results persist.  More precisely, we conclude this section with the following open questions.
\end{rem}

\begin{quest}
Let $\G$ be a non-unimodular discrete quantum group.  Do we then have:
\begin{enumerate}
\item $C(\hG)$ is nuclear $\implies$ $\G$ is amenable?  $L^\infty(\hG)$ is injective $\implies$ $\G$ is amenable?
\item $C(\hG)$ has the CBAP $\implies$ $\G$ is weakly amenable?  $L^\infty(\hG)$ has the w$^\ast$CBAP $\implies$ $\G$ is weakly amenable?
\item $L^\infty(\hG)$ has the Haagerup approximation property $\implies \G$ has the Haagerup property?
\end{enumerate}
\end{quest}

\section{Central approximation properties for discrete quantum groups} \label{CAP}

Let $\G$ be a (discrete) quantum group.  A multiplier $a \in MA(\G)$ is called \textit{central} if $a \in ZL^\infty(\G)$, the center of $L^\infty(\G)$.  I.e., if $L^\infty(\G) = \prod_{\pi \in \text{Irr}(\hG)} \mc B(H_\pi)$, and $p_\pi \in \mc B(H_\pi)$ is the central projection corresponding to the $\pi$th component of $L^\infty(\G)$.  Then $a \in ZMA(\G) = MA(\G) \cap ZL^\infty(\G)$ if and only if 
\[
a = (c_\pi p_\pi)_{\pi \in \text{Irr}(\hG)} \qquad (\text{with }c_\pi \in \C).
\] 
In a similar fashion, we define $ZM_{cb}A(\G):=M_{cb}A(\G)\cap ZL^\infty(\G)$.
   
 When studying approximation properties for $\G$ such as amenability, weak amenability, or the Haagerup property, we can impose an additional requirement on the Fourier multipliers that implement these properties.   Namely that they be central.   This idea naturally leads to the notion of {\it central approximation properties} for discrete quantum groups.  More precisely, we have the following definition.

\begin{defn}
A discrete quantum group $G$ is {\it centrally amenable/centrally weakly amenable/has the central Haagerup property} if: 
\begin{enumerate}
\item $\G$ is amenable/weakly amenable/has the Haagerup property, and
\item there exists a net $(a_i)_i \subset ZM_{cb}A(\G)$ that implements the corresponding approximation property from (1).
\end{enumerate} 
\end{defn}

\begin{rem}
For a discrete quantum group $\G$
, we evidently have that the central version of an approximation property for $\G$ implies the corresponding non-central one.  In particular, in the case of weak amenability, we can define a {\it central Cowling-Haagerup constant}
\[
Z\Lambda_{cb}(\G): = \inf\{\limsup_i \|a_i\|_{M_{cb}A(\G)}: (a_i)_i \subset A(\G) \cap ZL^\infty(\G) \text{ is a left approx. unit for }A(\G)\},
\] 
and we have
\[Z\Lambda_{cb}(\G) \ge \Lambda_{cb}(\G).\]
\end{rem}
In Section  \ref{CAP-monoidal} we will see that there are amenable discrete quantum groups $\G$ that fail to be centrally amenable.  
However, in the unimodular case, it turns out that $\G$ has a given approximation property if and only if it has the corresponding central version.   

\begin{thm}[\cite{Fr13}] \label{unimodular}
Let $\G$ be a unimodular discrete quantum group.  Then:
\begin{enumerate}
\item $\G$ is centrally amenable $\iff$ $\G$ is amenable.
\item $\G$ is centrally weakly amenable $\iff$ $\G$ is weakly amenable (and $Z\Lambda_{cb}(\G) = \Lambda_{cb}(\G)$).
\item $\G$ has the central Haagerup property $\iff$ $\G$ has the Haagerup property.
\end{enumerate}
\end{thm}

\begin{proof}
The key idea is to start with an arbitrary element $a \in M_{cb}A(\G)$ and construct from $a$ a new multiplier $\tilde{a} \in ZM_{cb}A(\G)$ that retains the ``good'' properties of the original multiplier (i.e., positive definiteness, cb norm bounds, belonging to $C_0(\G)$ or $A(\G)$, ...).  With this in mind, fix $a \in M_{cb}A(\G)$ and consider the element $\tilde a  = (\tilde{a}^\pi)_{\pi \in \text{Irr}(\hG)}\in ZL^\infty(\G)$ given by 
\[
\tilde{a}^\pi = \frac{\text{Tr}_\pi(a^{\overline{\pi}})}{n(\pi)}p_\pi \qquad (\pi \in \text{Irr}(\hG)).
\] 
Let $\hat R$ denote the unitary antipode of $L^\infty(\hG)$.  If we fix our representatives  $(U^\pi)_{\pi \in \text{Irr}(\hG)}$ of $\text{Irr}(\hG)$ so that $U^{\bar \pi} = \overline{U^\pi}$.  Then $\hat R$ is a $\sigma$-weakly continuous $\ast$-antiautomorphism of $L^\infty(\hG)$ satisfying
\[\hat R(u_{ij}^\pi)= (u_{ji}^\pi)^* = u_{ji}^{\bar \pi} \qquad (\pi \in \text{Irr}(\hG), \ 1 \le i,j \le n(\pi)),\]

Define a $\sigma$-weakly continuous map 
\begin{align} \label{centralization}\tilde{L} \in \mc {CB}_\sigma(L^\infty(\hG)); \qquad \tilde{L} =  \hat \Delta^{-1} \circ E \circ ((\hat R \circ L^{(a)} \circ \hat R) \otimes \iota) \hat \Delta,
\end{align}
where $E: L^\infty(\hG) \overline{\otimes}L^\infty(\hG) \to  \hat \Delta( L^\infty(\hG))$ is the conditional expectation from \eqref{E-formula} and $L^{(a)}$ is the adjoint of the left multiplier $a$.    
  Now, since $\hat R$ is a $\ast$-antiautomorphism, \[\|\hat R \circ L^{(a)} \circ \hat R \|_{cb} = \|m_a\|_{cb},\] and therefore $\|\tilde{L}\|\le \|L^{(a)}\|_{cb}$.

We now claim that $\tilde{L} = L^{\tilde{a}}$, so $\tilde a \in ZM_{cb}A(\G)$ and $\|\tilde{a}\|_{M_{cb}A(\G)} \le \|a\|_{M_{cb}A(\G)}$.  To check this, we fix a matrix element $u_{ij}^\pi \in \mc O(\hG)$ and calculate 
\begin{align*}
\tilde{L}u_{ij}^\pi &= \hat \Delta^{-1} \circ E \circ (\hat R \circ m_a \circ \hat R \otimes \iota) \hat \Delta(u_{ij}^\pi)\\
&=\hat \Delta^{-1} \circ E\Big(\sum_{k,l}a^{\bar \pi}_{kl}R(u_{li}^{\bar\pi})\otimes u_{kj}^\pi\Big) \\
&= (\hat \Delta)^{-1} \circ E\Big(\sum_{k,l}a^{\bar \pi}_{kl}u_{il}^{\pi}\otimes u_{kj}^\pi\Big) \\
&=\sum_{k,l}\delta_{k,l}a^{\bar \pi}_{kl}\frac{u_{ij}^{\pi}}{n(\pi)}=\frac{ \text{Tr}(a^{\bar{\pi}})}{n(\pi)}u_{ij}^\pi\\
&=L^{(\tilde a)}u_{ij}^\pi.
\end{align*}

Now that we have a natural method for constructing elements of $ZM_{cb}A(\G)$ from elements of $M_{cb}A(\G)$, the rest of the proof is relatively straightforward:  One simply has to check that if a net $(a_i)_i \subseteq M_{cb}A(\G)$ implements an approximation property for $\G$ such as (weak) amenability/the Haagerup property, then the corresponding net 
$(\tilde{a}_i)_i$ will implement the central version of this approximation property.  We leave the details to the reader
\end{proof}

\subsection{Going beyond the unimodular case} It is natural to wonder whether it is possible  to generalize the centralization procedure for cb multipliers used in the proof of Theorem \ref{unimodular} to general discrete quantum groups.  Looking at the above proof, one can see that a crucial tool available in the unimodular case was the existence of a $\hat \varphi \otimes \hat \varphi$-preserving conditional expectation $E: L^\infty(\hG) \overline{\otimes} L^\infty(\hG) \to \hat \Delta(L^\infty(\hG))$, from which we could build a completely positive right inverse $\hat \Delta^\sharp := \hat \Delta^{-1} \circ E$ for the coproduct $\hat \Delta$. 
When $\G$ is not unimodular, it turns out that there can never exist such a $\hat \varphi \otimes \hat \varphi$-preserving conditional expectation $E$, since by \cite{Ta72}, this can happen if and only if $\hat \Delta$ intertwines the modular groups $\hat \sigma_t$ and $\hat \sigma_t \otimes \hat \sigma_t$ of $L^\infty(\hG)$ and $L^\infty(\hG) \overline {\otimes} L^\infty(\hG)$, respectively.  But for every discrete $\G$ ,  it turns out that  
\[
\hat \Delta \circ \hat \sigma_t = (\hat \tau_{t} \otimes \hat \sigma_{t})\hat \Delta,
\]
where $\hat \tau_t$ is the scaling group for $L^\infty(\hG)$, and from this formula it follows that $\hat \Delta$ intertwines $\hat \sigma_t$ and $\hat \sigma_t \otimes \hat \sigma_t$ if and only if $\hat \sigma_t = \hat \tau_t = \iota$  (i.e., $\G$ is unimodular!).

However, not all is lost in the non-unimodular case.  Following \cite[Equation (2.2)]{An06}, we can define a normal Haar state-preserving unital completely positive map $\hat \Delta^\sharp:  L^\infty(\hG) \overline{\otimes} L^\infty(\hG) \to \hat L^\infty(\hG)$ by the pairing
\[
(\hat \varphi \otimes \hat \varphi)([\hat\sigma_{i/2} \otimes \hat \sigma_{i/2}](x)\hat \Delta(a)) = \hat \varphi(\hat\Delta^\sharp(x)\hat\sigma_{-i/2}(a)) \qquad (a \in \mc O(\hG) \odot \mc O(\hG), \ x \in \mc O(\hG)).
\]
Unraveling this definition, one obtains the following concrete formula
\[
\hat \Delta^\sharp(u_{ij}^\pi \otimes u_{k,l}^\sigma) = \delta_{j,k}\delta_{\pi,\sigma}\frac{u_{il}^\pi}{d(\pi)} \qquad (\pi,\sigma \in \text{Irr}(\hG), \ 1 \le i,j \le n(\pi), \ 1 \le k,l \le n(\sigma)).
\]
Now, if we use the map $\hat \Delta^\sharp$ in place of $\hat \Delta^{-1} \circ E$ in formula \eqref{centralization}, we obtain from any $a \in M_{cb}A(\G)$ a new element $\tilde a \in ZM_{cb}A(\G)$ with 
\[
\tilde{a}^\pi = \frac{\text{Tr}(a^{\bar \pi})}{d(\pi)}p_\pi \qquad (\pi \in\text{Irr}(\hG)) \quad \& \quad \|\tilde a\|_{M_{cb}A(\G)} \le \|a\|_{M_{cb}A(\G)}.
\] 
Observe that the only difference between this formula and the one in the unimodular case is that the classical dimension $n(\pi)$ is replaced by the quantum dimension $d(\pi)$.  This difference has the unfortunate effect of producing only central multipliers which are ``uniformly far'' from the identity multiplier $1 \in M_{cb}A(\G)$.  Indeed, if $\|a\|_{M_{cb}(\G)} \le C$, then $\|\tilde{a}^\pi\|_\infty \le C\frac{n(\pi)}{d(\pi)}$.  In conclusion, this approach to constructing central multipliers in the non-unimodular setting for the purpose of proving central approximation properties is destined to fail.  

In the next section we will see an alternate approach to establishing central approximation properties for non-unimodular discrete quantum groups using tools from monoidal equivalence for compact quantum groups.

\section{Central approximation properties and monoidal equivalence} \label{CAP-monoidal}

We continue our discussion of central approximation properties for discrete quantum groups by connecting these properties with the powerful notion of {\it monoidal equivalence} for compact quantum groups.  This remarkable connection between these seemingly unrelated concepts has had a profound impact on the operator algebraic theory of compact quantum groups, and has led (for instance) to the proof of central weak amenability and central Haagerup property for all free quantum groups and the duals of quantum automorphism groups of finite dimensional C$^\ast$-algebras.  See \cite{dCFY} and Section \ref{app}. 

We begin with the definition of monoidal equivalence.

\begin{defn}[Bichon-De Rijdt-Vaes \cite{BiDeVa06}]
Let $\G_1, \G_2$ be two compact quantum groups.  We say that $\G_1$ and $\G_2$ are {\it monoidally equivalent}, and write $\G_1 \sim^{mon} \G_2$, if there exists a bijection \[\Phi:\Irr(\G_1) \to \Irr(\G_2)\]  together with linear isomorphisms 
\[\Phi: \Mor(\pi_1 \otimes \ldots \otimes \pi_n, \sigma_1 \otimes \ldots\otimes \sigma_m) \to  \Mor(\Phi(\pi_1)\otimes\ldots \otimes \Phi(\pi_n), \Phi(\sigma_1) \otimes \ldots \otimes \Phi(\sigma_m)) \]
such that $\Phi(1_{\G_1}) = 1_{\G_2}$ ($1_{\G_i}$ being the trivial representation of $\G_i$), and such that for any morphisms $S,T$, 
\begin{align*}
\Phi(S \circ T) &=\Phi(S)\circ \Phi(T) \quad (\text{whenever $S \circ T$ is well-defined})\\
\Phi(S^*)&=\Phi(S)^* \\
\Phi(S \otimes T) &=\Phi(S) \otimes \Phi(T).
\end{align*}
\end{defn} 
A monoidal equivalence between $\G_1$ and $\G_2$ means that the abstract monoidal C$^\ast$-tensor categories $\text{Rep}(\G_i)$  (consisting of all finite dimensional Hilbert space representations of $\G_i$ equipped with their morphism spaces $ \Mor(\cdot, \cdot)$ and the tensor product operation $\otimes$) are isomorphic.  

As a first (and very important) example, we consider monoidal equivalences between free orthogonal quantum groups $O^+_F$, where $F \in \text{GL}_N(\C)$ with $N \ge 2$ and $F \bar F  \in \R 1$.

\begin{thm}[\cite{BiDeVa06}]  For $i=1,2$, fix $F_i \in \text{GL}_{N_i}(\C)$ with $F_i\bar F_i = c_i1$ and $c_i \in \R$.  Then the corresponding free orthogonal quantum groups $O^+_{F_1}$ and $O^+_{F_2}$ are monoidally equivalent if and only if \[\frac{c_1}{\text{Tr}(F_1^*F_1)}  =  \frac{c_2}{\text{Tr}(F_2^*F_2)}. \]
 Moreover, any compact quantum group $\G$ such that $O^+_{F_1} \sim^{mon}\G$ is isomorphic to an $O^+_{F_2}$ of the above form.   
\end{thm}  

\begin{proof}[Sketch]
The basic idea (going back to Banica's seminal work \cite{Ba96}) is that if $U_i$ is the fundamental representation
of $O^+_{F_i}$, then the defining relations for this quantum group are equivalent to the requirement that $U_i$ be irreducible and self-conjugate ($U_i = F_i\bar{U_i}F_i^{-1})$, and that the vector \[t_{F_i} = \sum_{k=1}^{N _i}e_k \otimes F_ie_k \in \C^{N_i} \otimes \C^{N_i}\] belongs to $\Mor(1, U_i \otimes U_i)$.  

Now, since $O^+_{F_i}$ is the universal compact quantum group with the above properties, the Tannaka-Krein-Woronowicz reconstruction theorem \cite{Wo88} implies that  $\text{Rep}(O^+_{F_i})$ must then be generated as a monoidal C$^\ast$-tensor category by the object $U_i$ together with the basic morphisms $\iota_{\C^{N_i}}$ and  $t_{F_i}$.   The monoidal equivalence $\Phi: O^+_{F_1} \sim^{mon} O^+_{F_2}$ can then be defined on these generators by setting
\[
\Phi(U_1) = U_2, \quad \Phi(\iota_{\C^{N_1}}) = \iota_{\C^{N_2}}, \quad \Phi\Big(\frac{1}{\text{Tr}(F_1^*F_1)^{1/2}}t_{F_1}\Big) = \frac{1}{\text{Tr}(F_2^*F_2)^{1/2}}t_{F_2}.
\] 
and extending to all of $\text{Rep}(O^+_{F_1})$ via the categorical operations.  Of course, one has to verify some details, but they can be found in \cite{BiDeVa06}.
\end{proof}

As a particular instance of the above theorem, recall that for $q \in [-1,1]\backslash \{0\}$, we have that \[SU_q(2) = O^+_{F_q} \quad \text{where} \quad F_q = \left( \begin{matrix} 0 &1\\
-q^{-1} & 0 \end{matrix}\right) \in \text{GL}_{2}(\C).\] 
Noting that $F_q \overline{F_q} = -q1$ and $Tr(F_q^*F_q) = 1+q^{-2}$, we immediately obtain the following corollary, which says that every free orthogonal quantum group $O^+_F$
is monoidally equivalent to a suitable $SU_q(2)$.  The result is quite remarkable, since $SU_q(2)$ is always co-amenable, while $O^+_F$ is generally not. 

\begin{cor} \label{mon-equiv-SU}
Let $F \in \text{GL}_N(\C)$ and $F\bar F = c1$ for $c \in \R$.  Then $O^+_F$ is monoidally equivalent to $SU_q(2)$ for the unique value of $q \in [-1,1]\backslash \{0\}$ satisfying $\text{Tr}(F^*F) = -c(q+q^{-1})$.
\end{cor}

\subsection{Central multipliers and monoidal equivalence}

We now make the connection between monoidal equivalence and  central approximation properties.  The key to everything is the following theorem of Freslon \cite{Fr13}  on transferring central completely bounded  multipliers between duals of monoidally equivalent compact quantum groups.  

\begin{thm}[Freslon \cite{Fr13}] \label{thm:mon-equiv-ap}
Let $\G_1,\G_2$ be discrete quantum groups and assume that $\Phi: \hG_1 \sim^{mon}\hG_2$ is a monoidal equivalence.  Given any central function $a_1 = (a^\pi p_\pi)_{\pi \in \Irr(\hG_1)} \in ZL^\infty(\G_1)$, define $a_2 =(a^\pi p_{\Phi(\pi)})_{\pi \in \Irr(\hG_1)} \in ZL^\infty(\G_2)$.
Then we have that $a_1 \in ZM_{cb}A(\G_1)$ if and only if $a_2 \in ZM_{cb}A(\G_2)$, and   \[\|a_1\|_{M_{cb}A(\G_1)} = \|a_2\|_{M_{cb}A(\G_2)}.\]
Moreover, $a_1$ is a normalized completely positive definite function if and only if $a_2$ is.    
\end{thm}  

In order to prove this theorem, the key idea is to compare the cb multiplier norms for $a_1$ and $a_2$ using the algebraic/analytic structure of the {\it linking algebra} associated to the given monoidal equivalence.

\begin{thm}[Bichon-De Rijdt-Vaes \cite{BiDeVa06}] \label{link-algebra}
Let $\G_1 \sim^{mon} \G_2$ be monoidally equivalent compact quantum groups with monoidal equivalence $\Phi$.  Then 
\begin{enumerate}
\item There exists a unique unital $\ast$-algebra $\mc B:= \mc O(\G_1,\G_2)$ (the linking algebra) equipped with a faithful state $\omega$ and unitary elements $X^\pi \in \mc B(H_\pi, H_{\Phi(\pi ))}) \otimes \mc B$ for each $\pi \in \Irr(\G)$, satisfying
\begin{enumerate}
\item The matrix elements of $(X^\pi)_\pi$ form a linear basis of $\mc B$.
\item $X^\sigma_{13}X^\gamma_{23}(S \otimes 1) = (\Phi(S) \otimes 1) X^\pi$ for each $S \in  \Mor(\pi, \sigma \otimes \gamma)$.
\item $(\iota \otimes \omega)(X^\pi) = 0$ if $\pi \ne 1$.
\end{enumerate}
\item There exists  a unique pair of  commuting actions $\G_1 \curvearrowright^{\delta_1} \mc B \curvearrowleft^{\delta_2} \G_2 $  (i.e., unital $\ast$-homomorphisms  $\delta_1: \mc B \to \mc B \otimes \mc O(\G_1)$ and $\delta_2: \mc B \to \mc O(\G_2) \otimes \mc B$ such that \[(\delta_1 \otimes \iota)\delta_1 = (\iota \otimes \Delta_1)\delta_1, \quad  (\iota \otimes \delta_2)\delta_2 =  (\Delta_2 \otimes \iota) \delta_2 \] and
\[ \text{span}(\delta_1(\mc B) (1 \otimes \mc O(\G_1))) = \mc B \otimes \mc O(\G_1),  \text{span}(\delta_2(\mc B) (\mc O(\G_2)\otimes 1)=  \mc O(\G_2) \otimes \mc B.\]    These actions are given by the formulas
\begin{align}\label{action-formula}(\iota \otimes \delta_1)(X^\pi) = X^\pi_{12}U^\pi_{13} \qquad (\iota \otimes \delta_2)(X^\pi) = U^{\Phi(\pi)}_{12}X^{\pi}_{13} \qquad (\pi \in \text{Irr}(\G_1)).
\end{align}
\item The state $\omega:\mc B \to \C$ is invariant under the actions $\delta_1, \delta_2$.  I.e.,
\[(\iota \otimes \omega)\delta_2 = (\omega \otimes \iota)\delta_1 = \omega(\cdot)1.\]  
 \end{enumerate}
Denote  by $B_r$ the C$^\ast$-algebra generated in the GNS-representation with respect to $\omega$ and by  $B_u = C^*_u(\mc B)$ the universal enveloping C$^\ast$-algebra of $\mc B$.  Then $\delta_i$ extends uniquely to a coaction of C$^\ast$-algebraic CQGs on $B_r$ and $B_u$, respectively.
\end{thm}

In order to motivate what follows, it is helpful to keep in mind the trivial example of a monoidal equivalence $\G_1 \sim^{mon} \G_1$ given by the identity map.  In this case, $\mc O(\G_1,\G_1) = \mc O(\G_1)$, $\omega = \varphi_{\G_1}$ is the Haar state, and $\delta_i$ is just the left/right action of $\G_1$ on $\mc O(\G_1)$ induced by the coproduct $\Delta_{\G_1}$.  Thinking about the linking algebra $\mc B$ in terms of this special case will help make the formulas in the above proposition appear more natural.

\begin{proof}[Rough outline of Theorem \ref{link-algebra}]
As a vector space, we define \[\mc B = \bigoplus_{\pi \in \Irr(\G)} \mc B(H_\pi,H_{\Phi(\pi)})^*\] and we define natural elements $X^\pi \in \mc B(H_\pi, \otimes H_{\Phi(\pi)}) \otimes \mc B$ by $(\omega_\sigma \otimes \iota)X^\pi = \delta_{\pi, \sigma}\omega_\sigma$ for all $\omega_\sigma \in \mc B(H_\sigma, H_{\Phi(\sigma)})^*$.  We can then define a unique bilinear map $\mc B \times \mc B \to \mc B$ using the relation (b) from the statement of the  theorem.  The associativity
of this map follows from the calculation 
\begin{align*}
&(X_{14}^aX_{24}^b)X_{34}^c((S \otimes 1)T\otimes 1) \\
&= (\Phi((S \otimes 1)T)\otimes 1)X^x \\
&=X_{14}^a(X_{24}^bX_{34}^c)((S \otimes 1)T\otimes 1) \qquad (S \in \Mor(y,a \otimes b), \ T \in \Mor(x,y \otimes c)),
\end{align*} and the fact that maps of the form $(S \otimes 1)T$ linearly span $\Mor(x,a\otimes b \otimes c)$.

The involution $b \mapsto b^*$ on $\mc B$ is defined  by
\[(X^\pi)^*_{13}(\Phi(t) \otimes 1) = X^{\bar \pi}_{23}(t \otimes 1) \qquad (t \in \Mor(1, \pi \otimes \bar \pi)). \]
Equivalently, we have \[((\omega_{\xi, \eta} \otimes \iota)(X^\pi))^* = (\omega_{(\xi^*\otimes 1)t, (1 \otimes \eta^*)\Phi(\tilde{t})}\otimes \iota)X^{\bar \pi},\] where $\tilde t \in \Mor(1, \bar \pi \otimes \pi)$ is such that $(t^* \otimes 1)(1\otimes \tilde{t})  =1$. 

One can then check that the above constructed elements $X^x$ are unitaries with respect to this $\ast$-algebra structure.  One can also check that if $a = \sum_{\pi, i} (\omega_{\xi_i^\pi, f_i^\pi} \otimes \iota)X^\pi \in \mc B$, where $(f_i^\pi)_i$ is an orthonormal basis for $H_{\Phi(\pi)}$, then \[\omega(aa^*) = \sum_{\pi,i} \frac{1}{d(\pi)}\langle Q_\pi \xi_i^\pi |\xi_i^\pi\rangle \ge 0.\]   
Hence $\omega$ is a state.

Given the above description of $\mc B$, the construction of the actions $\delta_1, \delta_2$ according to \eqref{action-formula} is well-defined, and they automatically satisfy the claimed $\omega$-invariance condition.  This $\omega$-invariance in turn implies the existence of the extension of $\delta_i$ to a C$^\ast$-algebraic action of $\G_i$ on $B_r$.  The case of $\delta_i$ extending to an action on $B_u$ follows from standard considerations. See for example \cite[Lemma 2.13]{dC16}.  
\end{proof}

By symmetry, for every monoidal equivalence of compact quantum groups $\Phi:\G_1 \sim^{mon}\G_2$, we have the inverse monoidal equivalence \[\Phi^{-1}:\G_2 \sim^{mon} \G_1,\] and  we can repeat the above construction for $\Phi^{-1}$  which yields a corresponding $\ast$-algebra $(\tilde {\mc B}, \tilde \omega)$ with faithful state $\tilde \omega$.  Let $(Y^\pi)_{\pi \in \Irr(\G_1)}$ be the unitary matrices corresponding to the algebra $\tilde {\mc B} = \mc O(\G_2,\G_1)$.  Then one has a canonical $\ast$-anti-isomorphism \[\rho: B_r \to \tilde B_r; \qquad (\iota \otimes \rho)(X^\pi) = (Y^\pi)^* \qquad (\pi \in \text{Irr}(\G_1)).\]
Thus $\tilde B_r \cong B_r^{op}$.  See \cite{VaVe07} for details.

Before getting to the proof of Theorem \ref{thm:mon-equiv-ap}, we need one more important lemma due to Vaes and Vergnioux.  See \cite[Theorem 6.1]{VaVe07}.

\begin{lem} \label{VaVe-lem}
Let $\Phi:\G_1 \sim^{mon}\G_2$ be a monoidal equivalence with reduced linking algebras $B_r$ and $\tilde{B}_r$ as above.  Then there exists an injective $\ast$-homomorphism \[\theta:C(\G_1) \to \tilde B_r \otimes B_r; \quad \text{given by } (\iota \otimes \theta)U^\pi = Y^\pi_{12}X^\pi_{13} \in \mc B(H_\pi) \otimes \tilde{B_r} \otimes B_r \qquad (\pi \in \text{Irr}(\G_1)).\] 
\end{lem} 
\begin{proof}
A priori, $\theta$ is defined as a map from $\mc O(\G_1)$ to $\tilde{\mc B} \odot  \mc B$, and the fact that it is a $\ast$-homomorphism is a routine calculation based on the definitions of $\mc B$  and  $\tilde {\mc B}$. 
To obtain the extension to a $\ast$-homomorphism $\theta:C(\G_1) \to \tilde B_r \otimes B_r$, we just observe that $(\tilde \omega \otimes \omega) \theta = \varphi_{\G_1}$.  Therefore $\theta$ extends to the GNS completions.
\end{proof}

\subsection{Proof of Theorem \ref{thm:mon-equiv-ap}}
Fix a monoidal equivalence $\Phi:\hG_1 \sim^{mon} \hG_2$ and let $B_r$ be the corresponding reduced linking algebra.  For each $\pi \in \Irr(\hG_1)$, let $B_\pi \subset B_r$ denote the linear span of the coefficients of the unitary $X^\pi$ (the so-called {\it spectral subspace of $\pi$}), and observe that with respect to the $L^2(B_r,\omega)$-Hilbert space structure, $B_r$ is the orthogonal direct sum $\bigoplus^{\ell^2}_{\pi \in \text{Irr}(\G_1)}B_\pi$.  Finally, denote by  $P_\pi:B_r \to B_r$ the orthogonal projection onto the spectral subspace $B_\pi$, but viewed as a finite rank idempotent map on the C$^\ast$-algebra $B_r$.  

Now, observe that for each $\pi \in \text{Irr}(\hG_1)$, our action $\delta_1:B_r \to B_r \otimes C(\hG_1)$ from Theorem \ref{link-algebra} satisfies the equation \[(\iota \otimes L^{(p_\pi)}) \circ \delta_1 = \delta_1 \circ P_\pi,\]
where $L^{(p_\pi)} \in \mc CB(C(\hG_1))$ is the adjoint of the multiplier $p_\pi \in ZM_{cb}A(\G_1)$.  By continuous linear extension, this yields the equality of (possibly unbounded)  linear maps
\[
(\iota \otimes L^{(a_1)}) \circ \delta_1 = \delta_1 \circ \Big( \sum_{\pi \in \text{Irr}(\hG_1)}a^\pi P_\pi \Big),
\]  for any $a_1 = (a^\pi p_\pi)_{\pi \in \text{Irr}(\hG_1)} \in ZL^\infty(\G_1)$.  But $\delta_1$ is an injective (hence completely isometric) $\ast$-homomorphism, hence we obtain we get 
\begin{align}\label{one}\Big\|\sum_{\pi \in \text{Irr}(\hG_1)}a^\pi P_\pi
\Big\|_{\mc {CB}(B_r)}  =\Big\|\delta_1 \circ \Big(\sum_{\pi \in \text{Irr}(\hG_1)}a^\pi P_\pi
\Big)\Big\|_{\mc {CB}(B_r)}= \|(\iota \otimes L^{(a_1)}) \circ \delta_1\|_{cb} \le \|a_1\|_{M_{cb}A(\G_1)}.
\end{align}
On the other hand, Lemma \ref{VaVe-lem} provides an injective $\ast$-homomorphism  $\tilde\theta:C_r(\G_2) \to B_r \otimes \tilde B_r$, which, by construction, satisfies 
\[\tilde \theta \circ L^{(p_{\Phi(\pi)})} = (P_\pi \otimes \iota) \tilde \theta \qquad (\pi \in \text{Irr}(\hG_1)) \implies \tilde \theta \circ L^{(a_2)} = \Big( \Big( \sum_\pi a^\pi P_\pi \Big)  \otimes \iota\Big) \tilde \theta, \]
where $a_2 = (a^\pi p_{\Phi(\pi)})_\pi \in ZL^\infty(\G_2)$.
Therefore 
\begin{align} \label{two}\|a_2\|_{M_{cb}A(\G_2)} = \|\tilde \theta \circ L^{(a_2)}\|_{cb} = \Big\| \Big( \Big( \sum_\pi a^\pi P_\pi \Big)  \otimes \iota\Big) \tilde \theta\Big\|_{cb} \le \Big\|\sum_{\pi \in \text{Irr}(\hG_1)}a^\pi P_\pi
\Big\|_{\mc {CB}(B_r)}
\end{align}
From \eqref{one} and \eqref{two}, we obtain 
\[
\|a_1\|_{M_{cb}A(\G_1)} \ge \|a_2\|_{M_{cb}A(\G_2)},
\]
and  by switching the roles of $\G_1$ and $\G_2$, the reverse inequality must hold as well.  Finally, we observe that complete positive definiteness is preserved by the bijection $ZM_{cb}A(\G_1) \owns a_1 \longleftrightarrow a_2 \in ZM_{cb}A(\G_2)$, since $L^{(a_1)}$ is unital complete contraction if and only $L^{(a_2)}$ is.

\section{Application: Approximation properties for free quantum groups} \label{app}

In this final section, our goal is to apply some of the techniques of the previous sections to study the Haagerup property and weak amenability for some examples of discrete free quantum groups.  The main family of examples that we will consider are the free orthogonal quantum groups.  This class of examples turns out to be the key to understanding all discrete free quantum groups.   The main result we plan to discuss is the following:

\begin{thm}[\cite{dCFY}] \label{approx-prop-fqg} 
Let $N \ge 2$, $F \in \text{GL}_N(\C)$ with $F\bar F \in \R1$, and put $\F O_F = \widehat{O^+_F}$.  Then the discrete quantum group $\F O_F$ has the central Haagerup property and is centrally weakly amenable with $Z\Lambda_{cb}(\F O_F) = 1$  
\end{thm}

\begin{rem}
The quantum group $\F O_F$ is an example of a (discrete) free quantum group in the sense of \cite{VeVo13}.  In fact, as a consequence of a structure theorem of Wang \cite{Wa02} a general free quantum group turns out to be a free product of the form
\[
\G = \F O_{F_1} \ast \F O_{F_2}\ast \ldots \ast \F O_{F_k} \ast \F U_{Q_1} \ast \ldots \ast \F U_{F_l}, 
\]
where $\F U_{Q_i}:= \widehat{U^+_{Q_i}}$ for some $Q_i \in \text{GL}_{N_i}(\C)$ and $F_j \in \text{GL}_{N_j}(\C)$ $F_j \bar F_j \in \R 1$.

In fact, Theorem \ref{approx-prop-fqg} is true for any free quantum group $\G$ of the above form.  The key idea to proving the more general form of Theorem \ref{approx-prop-fqg} is to first prove it for the special case as stated above, then extend to the $\F U_Q$ case using a free product argument, and then get the general case with yet another free product argument. See \cite[Section 5]{dCFY} for details.  In summary, the study of $\F O_F$ with $F \bar F \in \R 1$ is the crucial one.
\end{rem}

\begin{rem}
It is also worth mentioning that Theorem \ref{approx-prop-fqg} has also been used together with monoidal equivalence and free product results from \cite{dRVa10} and \cite{dCFY} to obtain analogous results for the discrete duals of quantum automorphism groups of finite dimensional C$^\ast$-algebras.

These results for quantum automorphism groups have, in turn, been used to prove central approximation properties for the discrete duals of free wreath products of compact quantum groups by quantum automorphism groups.  See for example \cite{Lemeux, LeTa14, FiPi15}.  
\end{rem}

\subsection{Theorem \ref{approx-prop-fqg} in the unimodular case}

Let us first sketch the proof of  Theorem \ref{approx-prop-fqg} in the simplest case where $\F O_F$ is unimodular.  The results obtained  in this case will give us some intuition for what to expect when we consider the general case. 

Now, if $\F O_F$ is unimodular, then (up to isomorphism of compact duals) this corresponds to only two cases: the first one being where $F = 1 \in \text{GL}_N(\C)$, or $N = 2l$ and $F = \text{diag}(\underbrace{F_0, \ldots, F_0}_{l \text{ times }})$, where $F_0 = \left[ \begin{matrix}
0 & 1 \\
-1 &0
\end{matrix}\right]$.  See \cite[Remark 4.3]{FiVe15} and \cite[Remark 5.7]{BiDeVa06}.   

We begin with a useful lemma.

\begin{lem} \label{characterize-cppd}
Let $\G$ be a unimodular discrete  quantum group and let $\mc ZC^u(\hG) \subseteq C^u(\hG)$ be the C$^\ast$-subalgebra generated by the irreducible characters $\{\chi_\pi\}_{\pi \in \text{Irr} (\hG)} \subset C^u(\hG)$.  Then a central element $a = (a^\pi p_\pi)_{\pi \in \text{Irr}(\hG)} \in ZL^\infty(\G)$ is a normalized completely positive definite function if and only if there exists a state $\psi \in \mc ZC^u(\G)^*$ such that 
\[
a^\pi = \frac{\psi(\chi_{\pi})}{n(\pi)} \qquad (\pi \in \text{Irr}(\hG)).
\]
\end{lem}

\begin{proof}
If $a= (a^\pi p_\pi)_{\pi \in \text{Irr}(\hG)} \in ZM_{cb}A(\G)$ is a normalized completely positive definite function, then the corresponding map 
\[
L^{(a)}: \mc O(\hG) \to \mc O(\hG); \qquad L^{(a)}u_{ij}^\pi = a^\pi u_{ij}^\pi \qquad (1 \le i,j \le n(\pi)), \ \pi \in \text{Irr}(\hG))
\]   is unital and completely positive.  It then follows that the linear functional  
\[
\tilde \psi:\mc O(\hG) \to \C; \qquad \tilde \psi = \hat \epsilon \circ L^{(a)}.
\]
is a state, and hence extends to a  state $\tilde \psi \in C^u(\hG)^*$. 
Setting $\psi  = \tilde\psi|_{\mc ZC^u(\hG)}$ (which is a state on $\mc ZC^u(\hG)$), we obtain 
\[
\psi(\chi_\pi) = \sum_{i=1}^{n(\pi)} \tilde{\psi}(u_{ii}^\pi) = \sum_{i=1}^{n(\pi)} \hat \epsilon (a^\pi u_{ii}^\pi) = n(\pi) a^\pi \qquad (\pi \in \text{Irr}(\hG)).
\]
Conversely, let $\psi \in \mc ZC^u(\hG)^*$ be a state.  Then we can extend $\psi$ to a state
$\tilde\psi \in C^u(\hG)^*$ (by the Hahn-Banach theorem). Let $\tilde{a} = (\tilde{a}^\pi)_{\pi \in \text{Irr}(\hG)}  \in L^\infty(\G)$ be given by 
\[
\tilde{a}^\pi_{ij} = \tilde{\psi}((u^\pi_{ji})^*) \qquad (1 \le i,j \le n(\pi)), \ \pi \in \text{Irr}(\hG)).
\]
Then $\tilde a$ is a normalized completely positive definite function on $\G$ (since $\psi \circ \hat S$ is a state and $L^{(\tilde a)} = ((\tilde\psi \circ \hat S) \otimes \iota)\hat \Delta$ is unital and completely positive on $\mc O(\hG)$, and this map is readily seen to factor through the quotient $C^u(\hG) \to C(\hG)$). 

Applying the averaging argument in the proof of Theorem \ref{unimodular} to the multiplier $\tilde{a}$, we obtain a new central normalized completely positive definite function $a = (a^\pi)_{\pi \in \text{Irr}(\hG)}$
given by 
\[
a^\pi = \frac{\text{Tr}(\tilde a^\pi)}{n(\pi)} = \frac{\psi(\chi_\pi^*)}{n(\pi)}  = \frac{\psi(\chi_{\bar \pi})}{n(\pi)} \qquad (\pi \in \text{Irr}(\hG)).
\]
\end{proof}

So, in order to prove the central Haagerup property for $\F O_F$ in the unimodular case, we need to get our hands on $\mc Z C^u(O^+_F)$.  For this, we recall a few facts from  Banica's seminal work \cite{Ba96}:  

\begin{thm}[\cite{Ba96}]  \label{banica-result}Let $F \in \text{GL}_N(\C)$ with $F\bar F = \pm 1$.  There is a maximal family $(U^k)_{k \in \N_0}$ of pairwise inequivalent irreducible unitary representations of $O^+_F$ such that
\[
U^0 = 1, \quad U^1 = [u_{ij}]_{1 \le i,j \le N} \quad \text{ (the fundamental representation)},
\]
and for $k \ge 1$, we have 
\begin{align} \label{fusion-rules}\overline{U^k} \cong U^k \quad \& \quad 
U^1 \otimes U^k \cong U^{k+1} \oplus U^{k-1} \cong U^k \otimes U^1.
\end{align}
Moreover, if $(\mu_k)_{k \in \N_0}$ denote the (dilated) type 2 Chebychev polynomials determined by the initial conditions and recursion
 \begin{eqnarray} \label{eqn_chebyshev}
\mu_0(x) = 1, \quad \mu_1(x) = x, \quad x\mu_k(x) = \mu_{k+1}(x) + \mu_{k-1}(x) && (k \ge 1),
\end{eqnarray} and $q \le q_0 \in (0,1]$ are defined so that 
\[
N = q_0 + q_0^{-1} \quad \& \quad \text{Tr}(F^*F) = q+q^{-1},
\] 
then we have the following (quantum) dimension formulas for $U^k$:
\[
n(k) = \mu_k(q_0 + q_o^{-1}) \quad \& \quad d(k) = \mu_k(q+q^{-1}) \qquad (k \in \N_0).
\]
\end{thm}

Using Theorem \ref{banica-result} we can obtain the following proposition.

\begin{prop} \label{char-algebra}
Let $N \ge 2$ and let $F \in \text{GL}_N(\C)$ with $F\bar F = \pm 1$ with $\F O_F$ unimodular.  Then there is a $\ast$-isomorphism 
\[
\mc ZC^u(O^+_F) \cong C([-N,N]) \quad \text{such that} \quad \chi_k \mapsto \mu_k|_{[-N,N]} \qquad (k \in \N_0),
\]
where $\chi_k$ is the character of the representation $U^k$ and $\mu_k$ is the $k$th Chebychev polynomial defined above.
\end{prop}

\begin{proof}
We have $\mc Z C^u(O^+_F) = C^*(\chi_k: k \in \N_0) = C^*(1 ,\chi_1)$, where the last equality follows from the self-conjugacy and fusion rules for the irreducible representations $(U^k)_{k \in \N_0}$ given by \eqref{fusion-rules}.  Moreover, $\chi_1$ is self-adjoint, so the Gelfand map yields a $\ast$-isomorphism \[\mc Z C^u(O^+_F) \to C(\sigma(\chi_1)); \quad 1 \mapsto 1, \ \chi_1 \mapsto \mu_1|_{\sigma(\chi_1)},\] 
where $\sigma(\chi_1)$ is the spectrum of $\chi_1 \in C^u(O^+_F)$.  Since we also have also $\chi_k = \mu_k(\chi_1)$ for each $k \ge 0$, it remains to show that $\sigma(\chi_1) = [-N,N]$.

To this end, we first note that $\sigma(\chi_1) \subseteq [-N,N]$, since $\chi_1 = \sum_{i=1}^N u_{ii}$ is a sum of $N$ contractions.  For the reverse inclusion, consider first the case where $F=1$.  Then there is a quotient map $C^u(O_F^+) \to C(O_N)$ (= the C$^\ast$-algebra of continuous functions on the orthogonal group $O_N$).  This map sends $\chi_1$ to the character $\chi_1^{O_N}$ of the fundamental representation of $O_N$, hence $\sigma(\chi_1) \supseteq \sigma(\chi_1^{O_N})$.  But this latter set is well known to be $[-N,N]$ -- since the trace of an $N\times N$ orthogonal matrix can be real number in $[-N,N]$.
For the other case, we have $N = 2l$ and $F = \text{diag}(\underbrace{F_0, \ldots, F_0}_{l \text{ times }})$, where $F_0 = \left[ \begin{matrix}
0 & 1 \\
-1 &0
\end{matrix}\right]$.  In this case,  one readily sees that we have a quotient map from $C^u(O^+_F)$ onto $C(\T)$ given by 
\[U^1 = U \mapsto V =  \Big\{\theta \mapsto  \text{diag}\Big(\underbrace{\left[\begin{matrix} \cos \theta & \sin \theta \\
-\sin(\theta) & \cos(\theta) \end{matrix}\right], \ldots, \left[\begin{matrix} \cos \theta & \sin \theta \\
-\sin(\theta) & \cos(\theta) \end{matrix}\right] }_{l \text{ times }}\Big) \Big\} \in M_{2l}(O_{2l}), \qquad (\theta \in \R)\]
In particular, $\chi_1$ gets sent to the function $\theta \mapsto f(\theta) = \text{Tr}(V(\theta)) = 2l\cos(\theta)$, whose spectrum is $[-2l,2l] = [-N,N]$, completing the proof.
\end{proof}

Combining Lemma \ref{characterize-cppd}  and Proposition \ref{char-algebra}, we immediately obtain the following characterization of completely positive definite functions on $\F O_F$ in the unimodular case.  This result was implicitly obtained in the $F=1$ case in \cite{Br12}.  See also \cite{CiFrKu14} for a nice treatment. 

\begin{thm} \label{Br12thm} \label{Khintchine}
Let $\F O_F$ be unimodular as above.  Then there is a one-to-one correspondence between
central normalized completely positive definite functions  $a=(a^k)_{k\in \N_0} \in ZM_{cb}A(\F O_F)$ and Borel probability measures $\mu$ on $[-N,N]$ given by 
\[a^k = \frac{\int_{[-N,N]}\mu_k(s)d\mu(s)}{\mu_k(N)} \qquad (k \in \N_0). \]
\end{thm}

We are almost ready to prove Theorem \ref{approx-prop-fqg} in the unimodular case.  The last ingredient we need is a remarkable result of Freslon, which gives polynomial bounds on the growth of the $M_{cb}A(\F O_F)$-norm of the sequence of finite rank central projections $p_k \in Z L^\infty(\F O_F)$ corresponding to the irreducible representation $U^k$ of $O^+_F$.  The following theorem should be thought of as a quantum group analogue of Buchholz's operator-valued Haagerup inequality on free groups \cite{Bu99} and the resulting cb-norm estimates for projections onto the space of convolution operators on free groups supported on words of a fixed length \cite[Section 9]{Pi03}. 

\begin{thm}[\cite{Fr12}] \label{freslon-cbnorm}
Let $N \ge 3$, $F \in \text{GL}_N(\C)$ with $F \bar F = \pm 1$, and let $0 < q <1$ be such that $\text{Tr}(F^*F) = q+q^{-1}$.  Let $p_k \in ZL^\infty(\F O_F)$  be the central projection corresponding to the $k$-th irreducible representation of $O^+_F$.  Then there is a constant $C(q) > 0 $ so that 
\[\|p_k\|_{M_{cb}A(\F O_F)} \le C(q)k^2 \qquad (k \ge 1).\]
\end{thm}

\begin{proof}[Proof of Theorem \ref{approx-prop-fqg} (unimodular case)]
We can assume $N \ge 3$, since $N=2$ corresponds to $SU_{\pm 1}(2)$, which are centrally amenable.  For $2 < s_0 <s < N$, consider the Dirac measure on $[-N,N]$ supported at $s$.  From Theorem \ref{Khintchine}, we obtain a net of normalized completely positive definite functions $(a_s)_{s \in (s_0 <s < N)} \subset ZM_{cb}A(\F O_F)$ such that 
\[a_s^k = \frac{\mu_k(s)}{\mu_k(N)} \qquad (k \in \N_0).\]

From the structure of the Chebychev polynomials $(\mu_k)_k$ it is possible to show (see \cite[Proposition 4.4]{Br12}) that there is a constant $D > 0$ such that 
\begin{align}\label{growth}0 < a^k_s \le D\Big(\frac{s}{N}\Big)^k \qquad (k \in \N_0, \ s_0 < s < N).
\end{align} Since we also have that $\lim_{s \to N}a_s = 1$ pointwise on $\N_0$, $(a_s)_{s_0 < s < N}$ is a bounded approximate identity for $C_0(\F O_F)$.  Therefore $\F O_F$ has the central Haagerup property.

To prove that $Z\Lambda_{cb}(\F O_F) = 1$, the idea is to use a standard technique of truncating the net $(a_s)_s \subset ZM_{cb}A(\F O_F)$ to obtain a net of finitely supported central elements of $A(\F O_F)$ with the following properties: (1) they still converge pointwise to $1$, and (2) we still have  control over their $M_{cb}A(\F O_F)$-norms.      

To this end, for each $s_0 < s < N$  and $n \in \N$ define
\[a_{s,n}  = \sum_{k=0}^n a^s_k p_k \in ZA(\F O_F).\] 
Then we have from Theorem \ref{freslon-cbnorm} and inequality \eqref{growth} 
\[\|a_s - a_{s,n}\|_{M_{cb}A(\F O_F)} \le \sum_{k=n+1}^\infty a^k_s\|p_k\|_{M_{cb}A(\F O_F)} \le \sum_{k=n+1}^\infty D\Big(\frac{s}{N}\Big)^k C(q)k^2.\] 
Hence \[\lim_{n \to \infty} \|a_s - a_{s,n}\|_{M_{cb}A(\F O_F)} \to 0 \qquad (s_0 < s < N).\]
Thus the net $(a_s)_{s_0 < s < N}$ can be approximated in the cb-multiplier norm by elements of $A(\F O_F)$.  It is then a routine exercise to extract a net $(a_i)_{i \in I}$ from the family $(a_{s,n})_{s,n}$ such that $a_i \to 1$ pointwise and $\limsup_i \|a_i\|_{M_{cb}A(\F O_F)}  = \sup_s\|a_s\|_{M_{cb}A(\F O_F)} = 1$.  In particular, $Z \Lambda_{cb}(\F O_F)=1$.
\end{proof}

\subsection{Theorem \ref{approx-prop-fqg} in the non-unimodular case}

We now consider a general $\F O_F$  with $F \in \text{GL}_N(\C)$, $N \ge 2$ and $F\bar F \in \R1$.  We start by recalling Corollary \ref{mon-equiv-SU}, which says that there is a unique
$q \in [-1,1] \backslash \{0\}$ such that the compact dual $O^+_F$ is monoidally equivalent to $SU_q(2)$.  Therefore it suffices by Theorem \ref{thm:mon-equiv-ap} to show that  $\widehat{SU_q(2)}$ has the central Haagerup property and $Z\Lambda_{cb}(\widehat{SU_q(2)}) = 1$.  But from the arguments in the previous section in the unimodular case, Theorem \ref{thm:mon-equiv-ap} already tells us that this is true for certain values of $q$.  More precisely, we have 

\begin{cor}[\cite{Fr13}] \label{kactononkac}
Let $q \in [-1,1] \backslash \{0\}$ be such that $q+q^{-1} \in \Z$.  Then $SU_q(2)$ is monoidally equivalent to a Kac type $O^+_F$ quantum group for some $F \in \text{GL}_N(\C)$ satisfying $F \bar F \in \R 1$ and $N = |q| + |q|^{-1}$.  In particular, $\widehat{SU_q(2)}$ has the central Haagerup property and $\Lambda_{cb}(\widehat{SU_q(2)}) = 1$.  Moreover, if $q \ne \pm 1$, the central Haageup property for $Z\widehat{SU_q(2)}$ is implemented by the net of central multipliers $(a_s)_{s_0 < s < |q| + |q|^{-1}} \in ZM_{cb}A(\widehat{SU_q(2)})$ given by 
\[a_s^k = \frac{\mu_k(s)}{\mu_k(|q| + |q|^{-1})} \qquad (k \in \N_0).\] 
\end{cor}

Interpreting the multipliers $(a_s)_s$ appearing in Corollary \ref{kactononkac} in terms of the states on $C(SU_q(2)) = C^u(SU_q(2))$  that they implement via the formula \[x \mapsto \psi_s(x) =  \epsilon (L^{(a_s)}(x)) \qquad (x \in \mc O(SU_q(2))),\] this suggests that for any $q \in (-1,1)\backslash \{0\}$, there ought to exist states $(\psi_s)_{s_0 <s < |q|+|q|^{-1}} \subseteq C(SU_q(2))^*$ defined by
\begin{align}\label{voigt-state}
\psi_s(u_{ij}^k) = \frac{\delta_{ij}\mu_k(s)}{\mu_k(|q|+|q|^{-1})} \qquad (\text{where } U^k  = [u_{ij}^k] \text{ is the $k$th irrep. of $SU_q(2)$}).
\end{align}

It turns out that such states do indeed exist, as was shown by De Commer, Freslon and Yamashita in \cite{dCFY}.  The main idea is to consider some $\ast$-representations of $C(SU_q(2))$ constructed by Voigt \cite{Vo11} in his study of the Baum-Connes conjecture for free orthogonal quantum groups.  Let us outline the main ideas, referring the reader to \cite{dCFY, Vo11} for the details.   

We begin by considering some representations of $\mc O(SU_q(2))$.  Given $z \in \C$, we can define an algebra homomorphism $\pi_z:\mc O(SU_q(2)) \to \text{End}(\mc O(SU_q(2)))$ by 
\[ 
\pi_z(x)y = x_{(1)}y (f_{1-z}\star S(x_{(2)})) \qquad (x,y \in \mc O(SU_q(2))),
\]  
where $f_{1-z}: \mc O(SU_q(2)) \to \C$ is the Woronowicz character and $f_{1-z}\star x = (\iota \otimes f_{1-z}) \Delta(x)$.
Using $\pi_z$, we define a family of linear functionals \[\omega_z: \mc O(SU_q(2)) \to \C; \quad 
\omega_z(x) = \varphi(\pi_z(x)1) =  \langle \pi_z(x)1|1 \rangle_{L^2(SU_q(2))}. \]

A direct calculation (see \cite[Proposition 10]{dCFY}) shows that for any matrix element $u_{ij}^k \in \mc O(SU_q(2))$ of the $k$th irreducible representation of $SU_q(2)$, we have
\begin{align*}
\omega_z(u_{ij}^{k}) &= \varphi(\pi_z(u_{ij}^{k})1) = \varphi((u_{ij}^{k})_{(1)}(f_{1-z}\star S((u_{ij}^{k})_{(2)})) ) \\
&= \frac{\delta_{ij}\mu_k(|q|^{\bar z} +|q|^{-\bar z})}{\mu_{k}(|q| +|q|^{-1})}. 
\end{align*}  
In particular, when $z = t \in (-1,1)$, the functional $\omega_t$ has the desired form of \eqref{voigt-state} (just take $\psi_s = \omega_t$ with $s = |q|^t + |q|^{-t}$).  Moreover, $\omega_t$ turns out to be a state in this case:  In \cite[Section 4]{Vo11}, it is shown that the unital $\ast$-subalgebra 
 \[\mc O(SU_q(2)/\T) = \{x \in \mc O(SU_q(2)): (\iota \otimes \pi_{\T})\Delta (x) = x \otimes 1\}\]
of polynomial functions on the {\it standard Podles Sphere} $SU_q(2)/\T$ is invariant under each homomorphism $\pi_z$.  (Here, $\pi_\T:C(SU_q(2)) \to C(\T)$ is the canonical quotient map associated to the quantum subgroup $\T$).  If $z = t \in (-1,1)$,  \cite[Lemma 4.3]{Vo11} shows that we can moreover deform the $L^2(SU_q(2))$ inner product on $\mc O(SU_q(2)/ \T)$ in a $t$-dependent way so that $\pi_t$ extends to $\ast$-representation of $C(SU_q(2))$ on $H_t$, the Hilbert space completion of $\mc O(SU_q(2)/\T)$ with respect to this new inner product.  From the structure of this deformed inner product, it follows (see \cite[Proposition]{dCFY}) that the following formula holds:
\[\omega_t(x) = \varphi(\pi_t(x)1)  = \la \pi_t(x)1|1 \ra_{H_t}  \qquad (x \in \mc O(SU_q(2))). \]
In particular, $\omega_t$ (and thus also $\psi_s$) is always a state.  Conclusion: $\widehat{SU_q(2)}$ has the central Haagerup property. 

To prove $Z\Lambda_{cb}(\widehat{SU_q(2)}) = 1$, the idea is exactly the same as in the unimodular case: we need to truncate our net of normalized completely positive definite functions coming from the central Haagerup property in a cb-norm controlled way.  Now, if $q$ is such that $SU_q(2)$ happens to be monoidally equivalent to an $O^+_F$ for which Theorem \ref{freslon-cbnorm} can be applied, then our truncation argument in the unimodular case carries over verbatim.  For the remaining range of values of $q$ (i.e., $|q| $ near $1$), we need another approach.

The alternate approach is in the spirit of the classical paper of Pytlik and Szwarc \cite{PySz86}, where one wants to show that the family of normalized completely positive definite functions $(b_t)_{t \in (0,1)} \subset ZM_{cb}A(\widehat{SU_q(2)})$ associated to the above states $\omega_t$ can be analytically continued to an analytic family \[\mc S \owns z \mapsto b_z \in ZM_{cb}A(\widehat{SU_q(2)}),\]
where $\mc S \subset \C$ is a domain containing $(-1,1) \subset \R$.  In \cite[Theorem 11]{dCFY}, this is shown to be possible.  It is moreover shown there that for $|z|$ sufficiently small, $b_z^3 \in A(\widehat{SU_q(2)})$.  But since the family $\mc S \owns z \mapsto b_z^3$ is analytic, it follows that \[b_z^3 \in \overline{A(\widehat{SU_q(2)}) }^{\|\cdot\|_{M_{cb}A(\widehat{SU_q(2)} )}} \qquad (z \in \mc S).\]
Using this last fact, together with the fact that $b_t^3 \in ZM_{cb}A(\widehat{SU_q(2)}) $ is a normalized completely positive definite function with $\lim_{t \to 1}b_t^3 = 1$, we easily deduce that $Z\Lambda_{cb}(\widehat{SU_q(2)}) = 1$.

\bibliographystyle{plain}

\bibliography{brannan-biblio}

\end{document}